\documentclass{amsart}

\usepackage{amsmath,amssymb,graphicx,epsfig}
\usepackage{amsmath}
\usepackage{amssymb}
\usepackage{framed}

\usepackage{graphics}
\usepackage{amsmath,amssymb,graphicx,epsfig}

\usepackage{amssymb,amsmath,amsfonts,amsthm,mathrsfs,graphicx,subfigure}
\usepackage{wrapfig}
\usepackage[mac]{inputenc}
\usepackage{fancyhdr} % ce paquet permet des hauts et bas de pages
\theoremstyle{break}
 \usepackage{xcolor}
 \usepackage{ulem}
 \usepackage{cancel}
\numberwithin{equation}{section}
 \newtheorem{thm}{Theorem}[section]
 \newtheorem{cor}[thm]{Corollary}
 \newtheorem{lem}[thm]{Lemma}
 \newtheorem{prop}[thm]{Proposition}
 \newtheorem{claim}[thm]{Claim}
 \newtheorem{defn}[thm]{Definition}
 \newtheorem{rem}[thm]{Remark}
\newtheorem{assume}[thm]{ASSUMPTION}
\newtheorem*{thma}{THEOREM A}

 \newcommand{\set}[1]{\left\{#1\right\}}

\newcommand{\eps}{\varepsilon}

\newcommand \be     {\begin{equation}}
\newcommand \ee     {\end{equation}}
\newcommand {\RR} {\mathbb{R}}

 \newcommand{\half}{\frac{1}{2}}

 \newcommand{\pa}{\partial}

 \newcommand{\suml}{\sum\limits}

 \newcommand{\ylambk}{y;\lambda,k}

\thanks{We thank Prof. Assia Benabdallah for her comments and suggestions that significantly contributed to  this paper. }

\date{\today}

\begin{document}
\setcounter{thm}{0}
\title{\textsf{\textbf{\Large
    CONCENTRATION AND NON-CONCENTRATION OF  EIGENFUNCTIONS OF SECOND-ORDER ELLIPTIC  OPERATORS WITH A DIVERGENCE FORM IN LAYERED MEDIA}}}
 \author{M. Ben-Artzi}
\address{Matania Ben-Artzi: Institute of Mathematics, The Hebrew University, Jerusalem 91904, Israel}
\email{mbartzi@math.huji.ac.il}
\author{Y. Dermenjian}
\address{Yves Dermenjian:  CNRS, Centrale Marseille, I2M, Aix Marseille Univ., Marseille, France}
\email{yves.dermenjian@univ-amu.fr}
\begin{abstract}
Let $\Omega'\subset\mathbb{R}^{d},\,\,d=1,2,\ldots$ be an open bounded smooth domain, and $$\Omega=\Omega'\times (0,H)\subset\mathbb{R}^{d}\times\RR_+.$$
The coordinates in $\Omega$ are designated as $x=(x',y)\in\Omega'\times (0,H).$

The paper deals with the concentration (and non-concentration) properties (in sectors of $\Omega$) of the eigenfunctions of the self-adjoint second-order elliptic operator  
 $$
 \aligned
%\begin{itemize}
A=-\nabla\cdot\tilde{c}\,\nabla \,\,\mbox{in}\,\, L^{2}(\Omega;\,dx) \,\,\mbox{with domain }\,\,D(A) = \{v\in H^1_0(\Omega) ; \tilde{c}\,\nabla v\in H^1(\Omega)\}.\endaligned $$

 The coefficient $\tilde{c}>0$ is assumed to be bounded, but no continuity assumption is imposed. It is analogous to the square of the speed of sound in the wave equation, and $\sqrt{\tilde{c}}$ is commonly known in the physical literature as  the \textit{celerity}. This study deals with \textit{layered media}, namely,  $\tilde{c}(x)$ depends only on the single spatial coordinate $y\in (0,H),$  so that $\tilde{c}(x)=\tilde{c}(x',y)=c(y). $ 
 
 The eigenvalues of $A$ are partitioned (apart from a small residual set) into two disjoint infinite sets. The corresponding  
  eigenfunctions are labeled as $\mathfrak{F}_{G}$\, (guided) and  $\mathfrak{F}_{NG}$ \,(non-guided). Their asymptotic properties are expressed by  suitable estimates as the associated eigenvalues tend to infinity. The eigenfunctions in $\mathfrak{F}_{G}$ concentrate in ``wells'' of $c(y),$ subject to polynomial rate of decay away from the concentration sector. The non-concentrating eigenfunctions in $\mathfrak{F}_{NG}$   are oscillatory in every sector with non-decaying amplitudes. These results hold uniformly for families of celerities with a common bound on their total variation.
  
  The paper leaves as an open problem the question of non-concentration in the case of a function $c(y)$ which is continuous but not of bounded variation.

\end{abstract}
\maketitle
\section{\textbf{INTRODUCTION}}
\label{section-introduction}
%%%%%%%%%%%%%%%%%%%%%%%%%%%%%%%%%%%%%%%%%%%%%%%%%%%%%%%%%%%%%%%%%
%%%%%%%%%%%%%%%%%%%%%%%%%%%%%%%%%%%%%%%%%%%%%%%%%%%%%%%%%%%%%%%%%
%

Let $\Omega'\subset\mathbb{R}^{d},\,\,d=1,2,\ldots$ be an open bounded smooth domain. In particular, the eigenfunctions of $-\Delta$ in $\Omega'$ (with homogeneous boundary conditions) form a complete basis in $L^2(\Omega').$ Our domain of interest is
\be\label{eqomega}\Omega=\Omega'\times (0,H)\subset\mathbb{R}^{d}\times\RR_+.\ee

 The coordinates in $\Omega$ are designated as $x=(x',y)\in\Omega'\times (0,H).$
 
Observe that the regularity assumption on $\Omega'$ can be considerably relaxed, but this is not the main thrust of the present paper.

 In this paper we consider the self-adjoint second-order elliptic operator  (details are given in Section ~\ref{secsetup} below)
 \be\label{eqoperator}
 \aligned
%\begin{itemize}
A=-\nabla\cdot\tilde{c}\,\nabla \,\,\mbox{in}\,\, L^{2}(\Omega;\,dx) \,\,\mbox{with domain }\,\,D(A) = \{v\in H^1_0(\Omega) ; \tilde{c}\,\nabla v\in H^1(\Omega)\}.
%\\
\endaligned
\ee

 It is a sequel to our paper ~\cite{BBD:5} where we studied the operator $-\tilde{c}\Delta.$ Even though this operator is close to the  one  in ~\eqref{eqoperator}, there are significant differences between them, forcing very different methods of proof. To mention just a few instances, we point out to the treatment of guided waves by trace estimates in Theorem ~\ref{theo-general-divergence-c-:1} or the modified convexity of non-guided eigenfunctions in Theorem ~\ref{thmcyirreg-minamp}.

We always assume homogeneous Dirichlet boundary conditions.

Our study deals with \textit{layered media}, namely, the \textit{celerity} $\sqrt{\tilde{c}}$ depends only on the single spatial coordinate $y\in (0,H),$  so that $\tilde{c}(x)=\tilde{c}(x',y)=c(y). $ In the study of the associated wave equation $c(y)$ has the physical meaning of the  square of the variable speed of sound.
%{\sout{ We use the terminology of \textit{diffusion coefficient} for lack of a better choice since it appears in the ``diffusive term''.

The reader is referred to ~\cite{grebenkov} for a survey of the geometrical structure of the eigenfunctions of the Laplacian, with very extensive bibliography.
   We refer to the recent paper ~\cite{lonzaga} for a study of the operator $A$ in the physical setting of the layered atmosphere. Some additional references to the physical literature will be given below.

The dependence of $\tilde{c}$ on a single coordinate results in studying the spectral properties of $A$ via an infinite set of ordinary differential operators with effective increasing potentials (See Remark ~\ref{reminfinitekl}).

 We  assume that
 \begin{equation}\label{eqassumptioncy}
\mbox{\bf(H)}\quad 0<c(y)\in L^{\infty}(\lbrack 0,H\rbrack), 0<c_{m}={\rm ess\, inf}\{c(y), y\in \lbrack 0,H\rbrack\}<c_{M}= {\rm ess\,sup}\{c(y), y\in \lbrack 0,H\rbrack\}.
 \end{equation}

  Fix $0<c_m<c_M,\,\,\eps>0.$ We consider the family of all functions $c(y)$ satisfying ~\eqref{eqassumptioncy}.
    \be\label{eq-def-K}\mathscr{K}=\set{c(y),\quad c_m={\rm ess\,inf}_{y} c(y)<c_M={\rm ess\,sup}_{y} c(y)}.\ee

 Generally speaking, the set of eigenfunctions is split into two categories: those composed of sequences of eigenfunctions (with increasing eigenvalues) involving concentration of  mass in proper subdomains of $\Omega,$ and those for which such concentration does not occur.

These two categories have been studied by physicists since a long time.  In general, the terminology used in the physical literature frequently refers to \textit{guided} or \textit{non-guided} waves, corresponding, respectively, to \textit{concentrating} or \textit{non-concentrating} modes. We shall use these terms interchangeably, as is appropriate in a particular context.

As far back as 1930, Epstein \cite{EP:1}  established (in unbounded domains) the existence of acoustic guided waves that are generalized eigenfunctions, i.e. not belonging to the domain of the operator, and are evanescent outside a ``guiding channel''. The underlying speeds were analytic functions depending on a single vertical coordinate. See ~\cite{Ped-Wh:1} for a more general study of Epstein's profiles. An extensive study of guided waves in the acoustic case can be found in ~\cite{Wil:1} and its bibliography.

We refer to the Introduction of our paper ~\cite{BBD:5} for a discussion  of  physical instances  (optical fibers, elasticity...) related to layered structure of the medium. 

%*********************************************************************************

The terms \textit{concentration} and \textit{non-concentration} do not always carry the same meaning when used by various authors. The following definition clarifies their meaning in this paper.

For an open set $\omega\subseteq\Omega$ and $v\in L^2(\Omega),$ define

 \be\label{eqdefRw}R_{\omega}(v) = \frac{\Vert v\Vert^{2}_{L^{2}(\omega)}}{\Vert v\Vert^{2}_{L^{2}(\Omega)}}.\ee
\begin{defn}
\label{def-concentration:1}
If $\set{v_j}_{j=1}^\infty\subseteq L^2(\Omega)$ is a sequence of normalized eigenfunctions associated with an increasing sequence of eigenvalues and
\be\label{eqconcentr} \lim\limits_{ j\to\infty}R_{\omega}(v_j)=0\ee
then we say that $\set{v_j}_{j=1}^\infty$ \textbf{concentrates} in $\Omega\setminus\omega.$

On the other hand, if
\be\label{eqnonconcentr}
\liminf\limits_{ j\to\infty}R_{\omega}(v_j)>0,\quad \forall\omega\subseteq\Omega,
\ee
  then the sequence is \textbf{non-concentrating.}

\end{defn}
\begin{rem}
\label{rem-def-concentration:1}
    Later on we shall extend these notions also to sets of eigenfunctions that are not necessarily arranged as such sequences. Note that we study concentration and non-concentration for {\upshape{infinite subsets of eigenfunctions} } and not necessarily for the {\underline{whole set} } of eigenfunctions.
\end{rem}

In general, the occurrence of concentration phenomena for second-order  operators of the types ~\eqref{eqoperator} depends  on two features:
\begin{itemize}
\item The shape of the boundary $\pa\Omega.$
 \item The geometric properties of the  coefficient  $\tilde{c}(x).$
 \end{itemize}
 The literature concerning the \textit{concentration/non-concentration phenomena} as related to the shape of $\Omega$ is very extensive. A well-known aspect is the connection of ``quantum ergodicity'' to ``classically chaotic systems'' ~\cite{BuZu:1,BuZw:1,Can-Galko:1,H-H-Marz:1,Marz:1} and references therein.  The paper ~\cite{grebenkov-1} deals with spherical and elliptical domains.

 In contrast,  in this paper we are interested in the effects of the layered medium.
   Thus it is more closely related to the study of operators of the type $L=-\nabla\cdot(c\nabla)+V$ on a finite domain, where the potential $V(x)\geq 0$ is positive on a subset of positive measure. Typically, eigenfunctions associated with eigenvalues below $ess\,sup\,V(x)$ are concentrating. In ~\cite{AFM:1} the authors replace $V$ by an \textit{effective potential} $u(x)$ satisfying $Lu=1.$ They show concentration and exponential decay of eigenfunctions as derived from the geometry of $u.$ These phenomena are linked to the universal mechanism for the  Anderson and weak localization ~\cite{filoche}.
   
   Our operator $A$ ~\eqref{eqoperator} does not involve a potential but the concentration of suitable sequences of eigenfunctions results from the geometry of the  coefficient $\tilde{c}(x).$ As we shall see in Theorem ~\ref{theo-general-divergence-c-:1} below there is a strong underlying geometric aspect;  the concentration expresses the fact that the masses of eigenfunctions ``flow'' (as the eigenvalues increase) into the ``wells'' (or ``valleys'').

 Turning to the {\underline{non-concentration case}}, we observe that the existing literature is less extensive, perhaps due to the fact that it is not directly related to physical or industrial applications. Nevertheless we shall see that  it leads to some interesting mathematical questions concerning the structure and asymptotics of  eigenfunctions (typically associated with large eigenvalues). Recent publications in this direction are ~\cite{H-Marz:1} dealing with non-concentration in partially rectangular billiards and ~\cite{Christ-Toth:1} concerning piecewise smooth planar domains.  A  non-concentration result in a stricter sense is that ``almost all eigenfunctions of a rational polygon are uniformly distributed'' ~\cite{rudnick}. Estimates for nodal sets such as ~\cite{DoFe:1}  were extended in ~\cite{JL:1,LaLe:1} motivated by questions from control theory and ~\cite{LLP:1} that deals with non-concentration in the Sturm-Liouville theory. In the 1-D case issues of non-concentration are closely related to details of oscillatory solutions. We shall come back to it later in this introduction.

This paper deals with  both  concentration and non-concentration phenomena  for eigenfunctions of layered operators. As already pointed out  the  latter is less studied in the literature, especially when the  celerity $\sqrt{c(y)}$ is not regular (even discontinuous). As a result, the non-concentration case plays a greater role in this paper.  For such eigenfunctions we extend the scope of the study; not only facts pertaining to non-concentration but a more detailed study of the structure of the solutions in terms of the oscillatory character, amplitudes and their ratios and asymptotic behavior. In contrast to the concentrating case, we shall see that the essential features of the non-concentrating solutions depend primarily  on the maximum and minimum of $\tilde{c}(x) = \tilde{c}(x',y) = c(y)$ and, going deeper into the structures, on the \textit{total variation} of $c(y).$

The paper is organized as follows.

In Section ~\ref{secsetup} we introduce all relevant notations and details concerning the functional setting.  In our case, the eigenvalues are classified by a double-index enumeration, with a conic sector (in index space {$(\mu_k^2,\,\lambda)$, see Figure \ref{image3}}). The set of eigenvalues   associated with concentrating eigenfunctions ($\mathfrak{F}_{G}$) (see ~\eqref{eqeigenconcentrate}) is distinguished from those (see ~\eqref{eqAceps})  associated with non-concentrating  eigenfunctions ($\mathfrak{F}_{NG}$). This  separation serves as the analog to the maximal value of a perturbation potential that separates concentrating from non-concentrating eigenfunctions in the potential perturbation framework.

Section \ref{section-divergence-gw:1} deals with guided waves for $A=-\nabla\cdot\tilde{c}\,\nabla .$
 The main result for the concentrating case ($\mathfrak{F}_{G}$) (Theorem \ref{theo-general-divergence-c-:1}) is proved.  The proof of the existence of sequences of eigenvalues satisfying the hypotheses of this theorem (see condition ~\eqref{eqeigenconcentrate}) is very similar to that in ~\cite{BBD:5} and will not be repeated here.   We mention that our proof yields only polynomial decay of the eigenfunctions outside the concentration layers, in contrast to the exponential decay obtained in the non-divergence case ~\cite{BBD:5}.

  In Section ~\ref{section-regular-ng} we turn to the case of non-concentrating eigenfunctions  (\textit{non-guided waves} in  the physical literature) for $A = -\nabla\cdot\tilde{c}\,\nabla.$ The set of corresponding eigenvalues is $ \mathscr{A}^{c}_\eps$ (see Definition ~\ref{defn-Ac}) that are located in the aforementioned upper conic sector in the index grid.

  The first approach that comes to mind is to transform the problem to a \textit{canonical form}. In other words, to use coordinate transformations so that the diffusion coefficient becomes a ``manageable'' perturbation of a constant one. In fact, this is done in this section under the assumption that $c(y)\in C^{2}(\lbrack 0,H\rbrack).$ In this case the classical Liouville transformation can be invoked, leading to a detailed asymptotic (almost sinusoidal) behavior of the non-concentrating eigenfunctions.

   In Section ~\ref{section-sufficient condition-ng} we study phenomena of non-concentration when  the  celerity $\sqrt{c(y)}$ is less regular.  The treatment is more delicate since the classical asymptotic methods are not applicable.  Instead, we introduce a powerful tool that we label as the  \textbf{minimal amplitude hypothesis} (see Definition ~\ref{definition-hypo-minam:1}).  This hypothesis is a geometric assumption on the asymptotic behavior  of the amplitudes in the $(u,\,u') $ phase plane.  The non-concentration of sets of oscillatory solutions follows directly from this geometric assumption (Theorem  ~\ref{theo-C1-noguided}). 
   
   Thus, in the rest of Section ~\ref{section-sufficient condition-ng} we focus on establishing the minimal amplitude hypothesis that implies Theorem ~\ref{thmcyirreg-minamp}. Furthermore, the hypothesis is established simultaneously for a full family $\mathscr{K}$ of coefficients (see ~\eqref{eq-def-K}).
       It underlines the fact that only the extremal values of $c(y)$ come into play for Lipschitz continuous coefficients. A different approach is employed in the case of piecewise constant coefficients. In each case, additional properties of the solutions are obtained. The ultimate case where we were able to establish the minimal amplitude hypothesis is for $c(y)$ being of \textbf{bounded total variation;} The non-concentration is shown to hold \textit{simultaneously} for the full family of  coefficients of total variation $TV(c)$ below a fixed $V.$ More specifically we get

   \begin{thma}
   \textit{
    Fix $0<c_m<c_M,\,\,\eps>0,\,\,V>0.$ Let
      $$\mathscr{K}_V=\set{c(y),\,\,c_m\leq c(y)\leq c_M,\,\,0<y<H,\,\,\,TV(c)\leq V}.$$
            Consider (for every $c(y)\in\mathscr{K}_V$) the subset of eigenvalues $ \mathscr{A}^{c}_\eps$ (see \eqref{eqAceps}) and the associated eigenfunctions $\set{v_\lambda}$ (of the form ~\eqref{equation-introduction:1bis} below).
      \\
             For an interval $(a,b)\subseteq (0,H) $ let $\omega:=\omega'\times(a,b)\subseteq\Omega,$ where $\omega'\subseteq\Omega'$ is an open set.
             \\
             If $\omega'\neq\Omega'$ assume that the family $\set{\phi_k(x')}_{k=1}^\infty$ of eigenfunctions of the Laplacian in $\Omega'\subseteq\RR^d$ does not concentrate in $\Omega'\setminus\omega'.$
             \\
             Then there exists $\mathfrak{f}_\omega>0$ such that
      \be\label{eqnoncontrateab}
      \mathfrak{f}_\omega\leq\frac{\Vert v_{\lambda}\Vert_{L^{2}(\omega)}}{\Vert v_{\lambda}\Vert_{L^{2}(\Omega)}}\leq 1
      \ee
      uniformly for all $c(y)\in\mathscr{K}_V$ and all eigenvalues in $ \mathscr{A}^{c}_\eps.$
      }
      \end{thma}

      This theorem will be proved as part of the more detailed Theorem ~\ref{prop-minam:1-new}.
        \begin{rem}\label{rem-uniform}
        The uniformity statement in $\mathscr{K}_V$ is relevant for physical applications, where the coefficient $c(y)$ is only approximately known.
        \end{rem}

    Remark that   the case of a continuous $c(y),$  but  not of bounded variation,  remains an open problem, whence the following question arises naturally:
\begin{framed}
\textbf{What degree of regularity of $c(y)$ could serve as necessary and sufficient in order to satisfy the minimal amplitude hypothesis (Definition ~\ref{definition-hypo-minam:1})?}
\end{framed}

\section{\textbf{SETUP AND MAIN RESULTS}}\label{secsetup}

%%%%%%%%%%%%%%%%%%%%%%%%%%%%%%%%%%%%%%%%%%%%%%%%%%%%%%%%%%%%%%%%%%%%%

We first represent the operator  $A= -\nabla\cdot\tilde{c}\,\nabla$  as a direct sum of ordinary differential operators.  For the Laplacian $-\Delta_{x'}$ acting in $L^{2}(\Omega')$ with domain $H^2(\Omega')\cap H^1_0(\Omega')$, we denote by $\set{(\mu_k^2,\phi_k)}_{k\geq1}$ the sequence of pairs (nondecreasing sequence of eigenvalues counting multiplicity, normalized eigenfunctions).  As the coefficient  function $\tilde{c}(x',y) = c(y)$ depends only on the last coordinate $y,$  a separation of coordinates  is natural. Using spectral decomposition in the $x'-$coordinate the operator $A=-\nabla\cdot\tilde{c}\,\nabla$ is unitarily equivalent to a direct sum of reduced operators in the form

\begin{equation}
\label{equation-spectralstructure:1}\
-\nabla\cdot\tilde{c}\,\nabla\approx \bigoplus_{k\in\mathbb{N}^*} A_{k}\quad \mbox{ acting in }\bigoplus_{k\in\mathbb{N}^*} L^{2}((0,H),dy),
\end{equation}

where

\be\label{eqAkofy}A_{k} :=c(y) \mu_{k}^{2} - \frac{d}{dy}c(y)\frac{d}{dy}, k = 1,2,\ldots,\mbox{ with }\,\,D(A_k)=\{u\in H^1_0(0,H), cu'\in H^1(0,H)\}.
\ee

 The eigenvalues of $A$ are ordered by a two-index system, namely $\sigma(A)=\{\beta_{k,\ell}, k,\ell\geq 1\}$ where $\Lambda_{k} = \{\beta_{k,1},\beta_{k,2}, \ldots\}$ is the increasing sequence  of the eigenvalues of $A_{k}.$
In others words, for each eigenvalue  $\lambda$  of $A$ there exists at least one $k\in \mathbb N^*$ such that $\lambda$ is a simple eigenvalue of $A_k,$ whence there exists at least a pair $(k,\ell)\in \mathbb N^*\times \mathbb N^*$ such that $\lambda =\beta_{k,\ell}$ (there is a one-to-one  relationship between the pairs $(k,\ell)$ and $(\lambda, k)$).
 If $\lambda$ is an eigenvalue, there is a finite number of pairs $(k,\ell)$ such that $\lambda=\beta_{k,\ell}.$

We construct an orthonormal basis  of eigenfunctions $\mathcal{B}=\set{v_{k,\ell}}_{k\geq 1,\ell\geq 1}$  associated with the eigenvalues $\beta_{k,\ell}.$  They are given by  $v_{k,\ell}(x',y)= \phi_{k}(x')u_{k,\ell}(y) $ where $u_{k,\ell}(y)$ satisfies

\begin{equation}
\label{equation-introduction:1}
 -(c(y)u_{k,\ell}')' + (c(y)\mu^{2}_{k}-\beta_{k,\ell})u_{k,\ell}= 0, \quad u_{k,\ell}(0) = u_{k,\ell}(H) = 0.
\end{equation}

\begin{rem}\label{reminfinitekl}
As is typical in ``separation of variables'' situations, the study of the spectral properties of the partial differential operator $A$ is carried out by controlling the behavior of the infinite set of ordinary differential operators of the type ~\eqref{equation-introduction:1}.
\end{rem}

\textbf{Henceforth we use the notation $u_{\lambda,k}$ instead of $u_{k,\ell}.$}
\\
 We often write $v_{\lambda}$ instead of $v_{k,\ell}.$
\begin{equation}
\label{equation-introduction:1bis}
\lambda = \beta_{k,\ell}\Longrightarrow v_{\lambda}(x',y)= v_{k,\ell}(x',y) = \phi_k(x')u_{\lambda,k}(y),\quad u_{\lambda,k}(y) \,\,\mbox{normalized in}\,\,L^{2}((0,H), dy).
\end{equation}
In this paper we are primarily interested in the phenomena of concentration or non-concentration of the mass of eigenfunctions.

\begin{defn}
\label{defn-layer:1}
For $a<b,$ a layer of $\Omega$ is defined as $\Omega_{a,b}:=\Omega'\times( a,b)\subseteq\Omega.$
\end{defn}
 %%%%%%%%%%%%%%%%%%%%%%%%%%%%%%%%%%%%%

 $\bullet$ {\textbf{ ON THE CONCENTRATION}}\\
 \vskip.5cm

 %%%%%%%%%%%%%%%%%%%%%%%%%%%%%%%%%%%%%%%%%%%%%%%%%%%%%%%%
 %%%%%%%%%%%%%%%%%%%%%%%%%%%%%%%%%%%%%%%%%%%%%%%%%%%%%%%%
\begin{minipage}[t]{154mm}
 \begin{wrapfigure}{r}{7cm}
\vskip-4.2cm
\centering
\includegraphics[scale=.33]{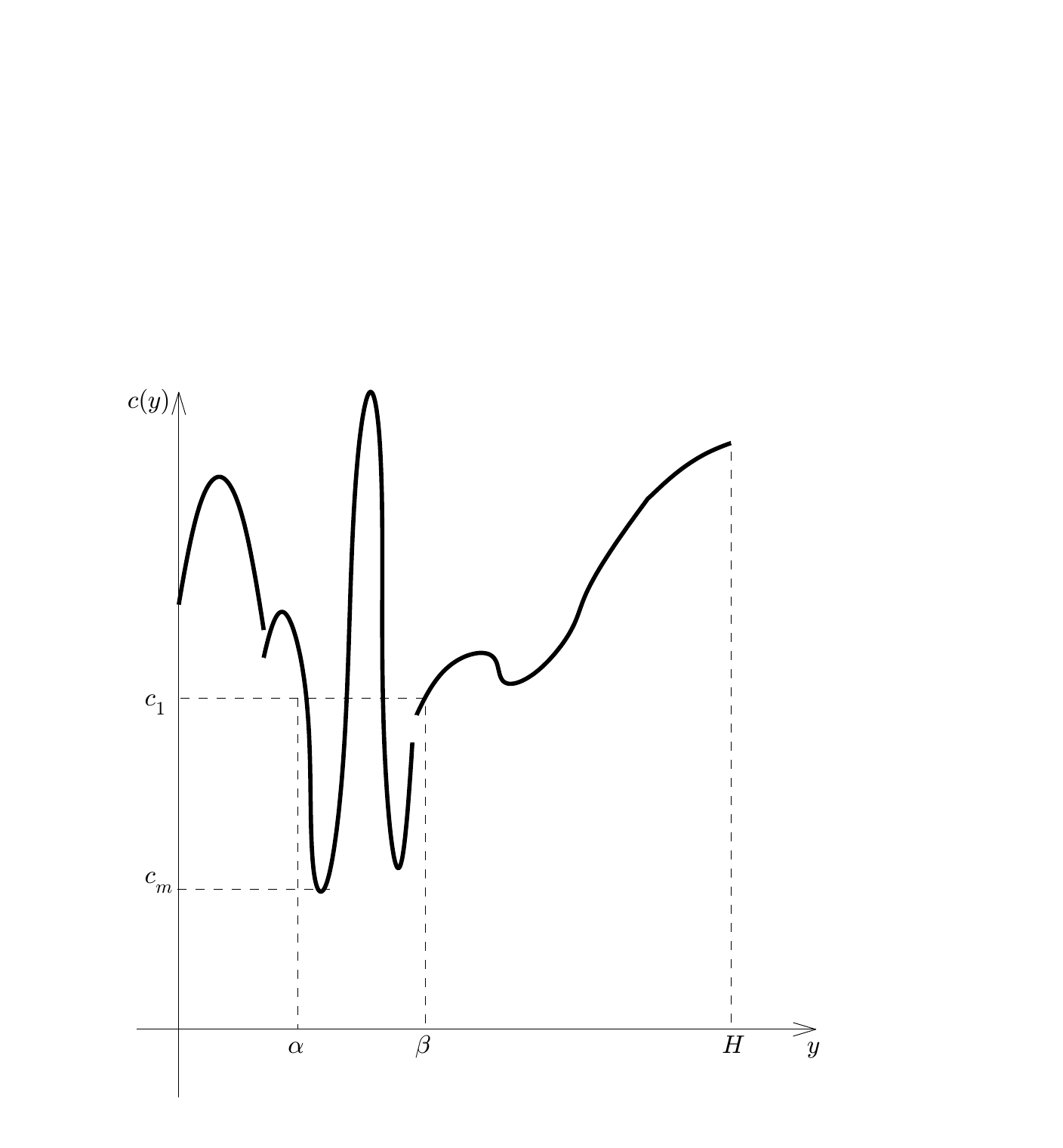}\hfill\\
\vskip-.5cm
\caption{\label{fig:image1}}
\end{wrapfigure}
\noindent
 \begin{defn}$\empty$\\
\label{def-well:1}
Let $\omega=\Omega_{\alpha,\beta}$ be a layer of $\Omega.$ We say that $\omega$ is a {\bf well for the profile} $c(y)$ if there exists $c_{1}>0$ such that
\begin{equation}
\label{cge:2}
\!\!\!\!\!\!\!\!\!\!\!\!\!\!\!\!\!\!\!\!\!\!\!\!\!\!\!\!\!\!\!\!\!\!\!\!\!\!\!\!\!\!\!\!\!\!\!\!\!\!\!\!\!\!\left\{\begin{array}{ll}
0<c_m
< c_1,\\
 c(y)\geq c_1>0,\quad a.e.\;y\in (0,H)\setminus (\alpha,\beta),
\end{array}
 \right.\end{equation}
 (See Figures ~\ref{fig:image1} and \ref{image2}).
\end{defn}
\end{minipage}
 %%%%%%%%%%%%%%%%%%%%%%%%%%%%%%%%%%%%%%%%%%%%%%%%%%%%%%%%
%%%%%%%%%%%%%%%%%%%%%%%%%%%%%%%%%%%%%%%%%%%%%%%%%%%%%%%%

\vskip2cm

 Note that our definition allows at most one well in the domain $\Omega,$ but a well may contain several minima (see Figure ~\ref{fig:image1}).

  In the concentration case we shall prove Theorem  ~\ref{theo-general-divergence-c-:1}, which yields  decay outside a well.  Unlike the case  of the operator  $-\tilde{c}\Delta$  ~\cite[Theorem 2.4]{BBD:5} we do not obtain here an exponential rate of decay outside the well.  Observe that the only hypothesis imposed on $c(y)$ is ~\eqref{eqassumptioncy}.

%%%%%%%%%%%%%%%%%%%%%%%%%%%%%%%%%%%%%

Our result here can be compared to the results of ~\cite{AFM:1},  that involves a positive potential perturbation. It leads to a construction of a ``landscape'' function that serves (rather, its inverse) as an ``effective potential''. In terms of this potential the exponential decay is     expressed by an  ``Agmon-type'' \cite{AG:1} metric. In our proof the 1-D dependence of $c(y)$ plays an important role.  A trivial shift leads to an operator of the form 
$A:=-\nabla\cdot\tilde{c}\,\nabla+V,$ where $V\equiv 1.$\\
Thus, the landscape function $\theta(x)$ is the function satisfying
\be\label{eqlandscape}
-\nabla\cdot\tilde{c}\,\nabla\theta+\theta=1,
\ee
subject to the Neumann boundary condition ~\cite[Section 3]{AFM:1}.
The projection of the landscape function $\theta$ on each fiber $L^2(0,H)$ yields a reduced "effective potential". Note that  the continuity of the coefficients is assumed, whereas  in our treatment the important case of a piecewise-constant $\tilde{c}$ is included.

\noindent

$\bullet$ {\textbf{ ON THE NON-CONCENTRATION.}}
 The second type of results concerns the  sets (indexed by $\varepsilon>0$) of non-guided normalized eigenfunctions (the set $\mathfrak{F}_{NG}$ of the Introduction). They are associated with eigenvalues
  $$\lambda= \beta_{k,\ell}> (c_{M}+\varepsilon)\mu_{k}^{2}, c_{M}:={\rm ess\,sup}_{y} c(y).$$
  This set is characterized
  by the fact that there is a positive lower bound for the masses in any layer $\Omega_{a,b},$  uniformly for all its elements.

     Recall  that $\lambda$ can correspond to several pairs $(k,\ell)$ and only some of them satisfy the above inequality.

  The geometrical interpretation of non-concentration is clear in the one-dimensional case $\Omega' = (0,L)$ and $\lambda= \beta_{k,\ell}> c_{M}\mu_{k}^{2}$: at each interface the angle between the wave and the normal is less than the critical angle stipulated by  geometric optics. So, the eigenfunction can travel across each layer without big loss.

  In physical applications it is conceivable that the diffusion coefficient $c(y)$ is known only approximately. It is therefore interesting to extend our study to deal with sets of such coefficients.  Let $0<c_m<c_M$ be fixed. We assume that every
 coefficient $c(y)$ satisfies condition $\mbox{\bf(H)}$ (see ~\eqref{eqassumptioncy}) and recall that we denote by
  $\mathscr{K}$
  the family of all such coefficients (see ~\eqref{eq-def-K}).

In various cases, we shall impose further assumptions on the elements of $\mathscr{K}.$

  We introduce a subset of eigenvalues,  whose associated eigenfunctions will be shown to be (perhaps under additional assumptions) non-concentrating.
\begin{defn}\label{defn-Ac} Fix $\eps>0.$ For any fixed $\mu_k,$ let $\ell_{0,k}$ be the first $\ell$ satisfying $\beta_{k,\ell_{0,k}}\geq (c_M+\eps)\mu_k^2.$\\
We designate
 (see Figure \ref{image3})
\be\label{eqAceps} \mathscr{A}^{c}_\eps={\underset{k=1}{\overset{\infty}\bigcup}}\set{(\mu_k,\lambda),\,\,\,\lambda= \beta_{k,\ell},\quad\ell\geq\ell_{0,k}}.
\ee
\end{defn}
Next we define the \textbf{minimal amplitude} of the family of the associated solutions (of ~\eqref{equation-introduction:1}) as follows.
\begin{equation}
 \label{equation-definition-hypo-minam:1}
 \mathfrak{r}_{c,\eps}^{2} = \inf_{\begin{array}{c} y\in \lbrack 0, H\rbrack,\\ (\mu_k,\lambda)\in\mathscr{A}^{c}_\eps\end{array}}\lbrack u_{\lambda,k}(y)^{2} + (c(y)u'_{\lambda,k}(y))^{2}\rbrack.
 \end{equation}
 \begin{defn}
 \label{definition-hypo-minam:1}
 Let $c(y)\in \mathscr{K}.$ We say that $c(y)$ satisfies the {\textbf{\textit minimal amplitude hypothesis}} with respect to $\mathscr{A}^{c}_\eps$  if
 \be\label{eqrcgreat0}\mathfrak{r}_{c,\eps}>0.\ee
 \end{defn}
\begin{rem}\label{remprufer}
 Note that this hypothesis has a very clear geometric interpretation by means of the Pr\"{u}fer substitution ~\cite{birkhoff}.
\end{rem}
 \noindent
Observe that while the minimal amplitude deals with the sum of squares $u_{\lambda,k}(y)^{2} + (c(y)u'_{\lambda,k}(y))^{2},$ the non concentration involves only the integral of $u_{\lambda,k}(y)^{2}$ over various intervals. The following theorem connects these topics, showing that the minimal amplitude hypothesis implies non-concentration. Here we state it using the physical model with the spectral parameter $\lambda.$ It is  proved in a somewhat more detailed form (using the reduced eigenfunctions $u_{\lambda,k}$) as Theorem ~\ref{thmcyirreg-minamp}.
\begin{thm}[Non-concentration in any layer]
\label{theo-C1-noguided}
Let $c(y)\in \mathscr{K}$ satisfy the minimal amplitude hypothesis. For any $(\mu_{k},\lambda)\in \mathscr{A}^{c}_\eps$  let $v_{\lambda}(x)$  be an associated eigenfunction.
  \\
 Let $\omega:=\omega'\times(a,b)\subseteq\Omega$ be an open set.
  If $\omega'\neq\Omega'$ assume that the family $\set{\phi_k(x')}_{k=1}^\infty$ of eigenfunctions of the Laplacian in $\Omega'\subseteq\RR^d$ does not concentrate in $\Omega'\setminus\omega'$ (see Definition \ref{def-concentration:1} and Remark \ref{rem-def-concentration:1}).
  \\
  Then  there exists  a constant $C_{\omega}>0$  such that,
\begin{equation}
\label{ineq-ng-general:1}
%\!\!\!\!\!\!\!\!\!\!\!\!\!\!\!\!\!\!\!\!\!\!\!\!
0< C_{\omega}\leq \inf\limits_ {(\mu_{k},\lambda)\in \mathscr{A}^{c}_\eps}\,\frac{\Vert v_{\lambda}\Vert_{L^{2}(\omega)}}{\Vert v_{\lambda}\Vert_{L^{2}(\Omega)}}\leq 1.
\end{equation}
\end{thm}
\begin{rem}\label{remuniformcy}
We shall see that in various cases we can find subsets $\mathscr{K}_1\subseteq\mathscr{K}$ such that the inequality ~\eqref{ineq-ng-general:1} holds uniformly with respect to $c\in \mathscr{K}_1.$
\end{rem}
%%%%%%%%%%%%%%%%%%%%%%%%%%%%%%%%%%%%%
\setcounter{figure}{1}
\hspace{-.5cm}\begin{minipage}[t]{149mm}
\begin{wrapfigure}{r}{6.1cm}
\vskip-0.3cm
\includegraphics[scale=.3]{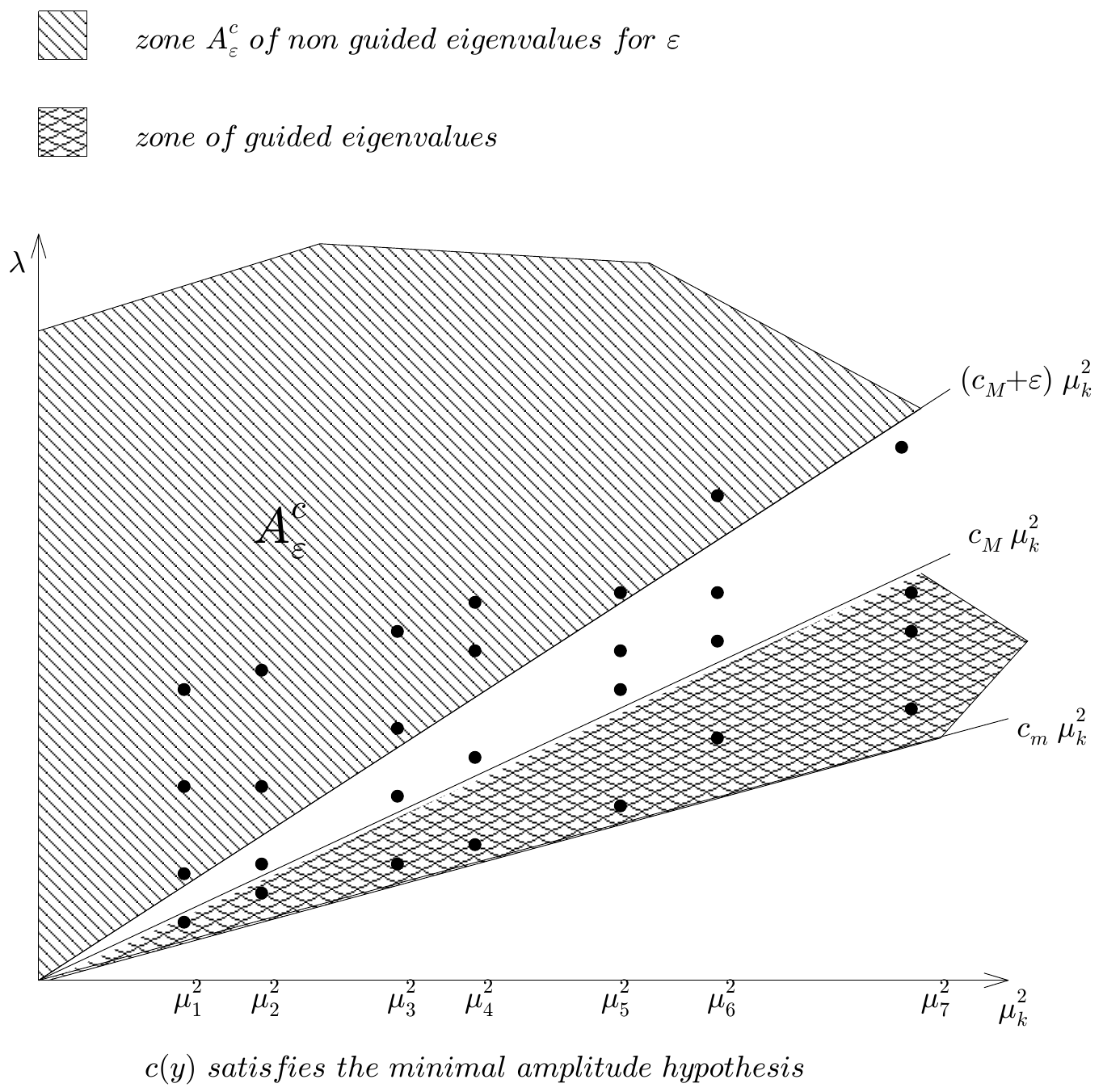}\vskip-.4cm\caption{\label{image3}}\hfill
\end{wrapfigure}
\noindent
\begin{rem}\label{rem-about-min-amplitude}$\empty$
 \begin{itemize}
 
 \item
  In Section ~\ref{section-sufficient condition-ng} we show that any eigenfunction associated with eigenvalues in $\mathscr{A}^{c}_\eps$ behaves in an oscillatory fashion. This is a straightforward consequence of the comparison principle. However, it does not exclude the possibility that some of the sections of the oscillatory solution may ``flatten out'', namely their amplitudes shrink as $\lambda\to\infty.$
  The condition ~\eqref{equation-definition-hypo-minam:1} ensures that such phenomena do not happen, as is stated in Theorem ~\ref{theo-C1-noguided}.
\item
 In Subsection ~\ref{subsec-more-minim} we discuss the meaning of the minimal amplitude hypothesis. If the function $c(y)$ is of bounded total variation we prove (Theorem ~\ref{prop-minam:1-new}) that it satisfies the hypothesis with respect to $ \mathscr{A}^{c}_\eps.$  This covers the cases of
 functions in $ C^{1}(\lbrack 0,H\rbrack)$ as well as  functions in $W^{1,1}(\lbrack 0,H\rbrack)$,  piecewise constant functions ...
 \end{itemize}
 \end{rem}
 \end{minipage}
\vskip.2cm
\noindent
%%%%%%%%%%%%%%%%%%%%%%%%%%%%%%%%%%%%%%%%%%%%%%

%%%%%%%%%%%%%%%%%%%%%%%%%%%%%%%%%%%%%%%%%%%%%%%%%%%%%
%%%%%%%%%%%%%%%%%%%%%%%%%%%%%%%%%%%%%%%%%%%%%%%%%%%%%
\section{\textbf{GUIDED WAVES: THE GENERAL CASE}}
\label{section-divergence-gw:1}
%%%%%%%%%%%%%%%%%%%%%%%%%%%%%%%%%%%%%%%%%%%%%%%%%%%%%
%%%%%%%%%%%%%%%%%%%%%%%%%%%%%%%%%%%%%%%%%%%%%%%%%%%%%

\begin{thm}[Concentration in the layer $\Omega_{\alpha,\beta}$]
\label{theo-general-divergence-c-:1}
 Let $\Omega_{\alpha,\beta}$   be a well for $c$  (as in Definition ~\ref{def-well:1}) and  let $ \lambda$   be an  eigenvalue of  $A_{k}$ (hence of $A$) such that
     
   \be\label{eqeigenconcentrate}c_m\,\mu_{k}^{2}\leq \lambda\leq (c_{1}-\varepsilon)\mu_{k}^{2}.
\ee
   
  Let $v_{\lambda}$ be the associated normalized eigenfunction.  If  $\overline{\Omega_{a,b}}\subseteq\bar{\Omega}\setminus\overline{\Omega_{\alpha,\beta}}$,
 then 

\begin{equation}
\label{equation-decreasing:1}
\underset{\lambda\to\infty}{\lim}\int_{\Omega_{a,b}}\vert v_{\lambda}(x)\vert^{2}\, dx = 0.
\end{equation}
\end{thm}

\setcounter{figure}{2}
\begin{figure}[h]
\centering
\includegraphics[scale=.25]{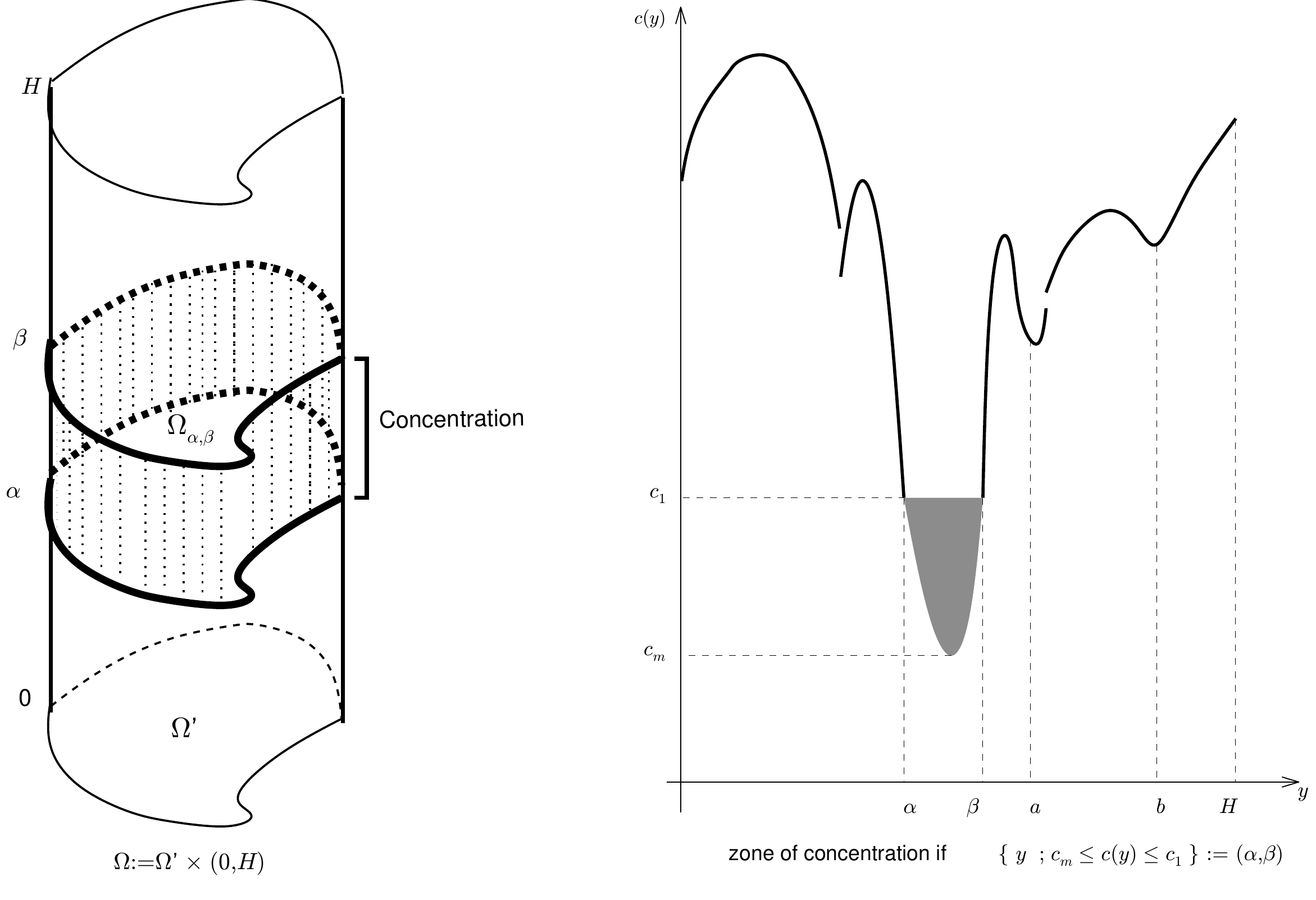}
\caption{}
\label{image2}
\end{figure}
\begin{proof}  Since there is no explicit formula for the Green function we  replace the method of proof of ~\cite[Theorem 2.4]{BBD:5} 
by
 an approach based on general trace estimates that   is applicable in the present case but also for $A=-\tilde{c}\,\Delta.$ So  let
$v_{\lambda}(x',y)= \phi_{k}(x')u_{\lambda,k}(y) $  where,  in analogy to ~\eqref{equation-introduction:1}, $u(y)=u_{\lambda,k}(y)$ satisfies

 \begin{equation}
\label{equation-premiereapproche:1}
\left\lbrace\begin{array}{l}
(cu')'= c(\mu_k^{2} -\frac{\lambda}{c})u \mbox{ on }(0,H),\\
u(0) = u(H) = 0.
\end{array}\right.
\end{equation}
 Since $\|u'\|_{L^{2}(0,H)}$ is an equivalent norm on $H^{1}_0(0,H)$  a standard trace estimate ~\cite[Chapter 1, Section 2, Proposition 2.3]{LionsMagenes} yields
\be\label{eqtraceuy}
 |u(y)|\leq\gamma \|u\|_{L^{2}(0,H)}^{1-s}\|u'\|_{L^{2}(0,H)}^s,\quad y\in (0,H),\,\,\frac12<s\leq 1.
\ee
Here and below $\gamma$ is a generic constant that may depend on various parameters
($s, H, c_m...$) but not on $u,\,y,\,\mu_k,\,\lambda.$

We define the self-adjoint operator (with same domain as $A_k$)
$$\mathscr{A}_k=A_{k}-c(y)\mu_{k}^{2} = -\frac{d}{dy}\big(c(y)\frac{d}{dy}\big).$$
The coercivity property $(\mathscr{A}_ku,u)_{L^{2}(0,H)}\geq c_m\|u'\|_{L^{2}(0,H)}^2$
entails, by the Cauchy-Schwarz inequality
$$\|u'\|_{L^{2}(0,H)}^2\leq \frac{1}{2c_m}\Big[\|\mathscr{A}_ku\|^2_{L^{2}(0,H)}+\|u\|^2_{L^{2}(0,H)}\Big].$$
 Invoking ~\eqref{equation-premiereapproche:1}, we obtain
  \be\label{equyu'L2}
  \|u'\|_{L^{2}(0,H)}^2\leq \frac{1}{2c_m}\Big[(1 + \xi_{M}^{2})\| u\|^{2}_{L^{2}(0,H)}\Big],\quad \xi_{M}^{2}= c_M\mu_{k}^{2} -\lambda.
  \ee
  Finally, combining this estimate with ~\eqref{eqtraceuy} we obtain
  \be\label{equpestuy}
  |u(y)|\leq\gamma(1 + \xi_{M}^{2})^{\frac{s}{2}}\|u\|_{L^{2}(0,H)},\quad \frac12<s\leq 1,\,\,y\in (0,H).
  \ee
  In particular, for $y=\beta,$\\

%%%%%%%%%%%%%%%%%%%%%%%%%%%%%%%%%%%%%%%%%%%%%%%%%%%%%

%%%%%%%%%%%%%%%%%%%%%%%%%%%%%%%%%%%%%%

\begin{equation}
\label{equation-div-upperbound:1}
\vert u(\beta)\vert^{2} \leq \gamma\, (1 + \xi_{M}^{2})^{s}\Vert u\Vert^{2}_{L^{2}(0,H)}.
\end{equation}
%%%%%%%%%%%%%%%%%%%%%%%%%%%%%%%%%%%%%%
\noindent
We next  obtain  a lower pointwise bound. Suppose that $\lbrack a,b\rbrack\subseteq (\beta,H).$
As $u(H) = 0$ we assume $(cu')(H)<0$ (otherwise replace $u$ by $-u$) and $u$ is monotone decreasing in $\lbrack \beta, H\rbrack$\footnote{

In view of ~\eqref{equation-premiereapproche:1} $(cu')(y)$ is strictly increasing in $(\beta,H)$ hence $cu'(y)<0$ there. }, whence
\begin{equation*}
\label{equation-lowerbound:2}
(cu')(y) = (cu')(H) - \int_{y}^{H}c(\sigma)\big(\mu^{2}_{k} -\frac{\lambda}{c(\sigma)}\big)u(\sigma)\,d\sigma\leq - \int_{y}^{H}\xi_{1}^{2}{c(\sigma)}u(\sigma)\,d\sigma, \quad \xi_{1}^{2}:= \mu_{k}^{2}-\frac{\lambda}{c_{1}}.
\end{equation*}
Multiplying by $2u(y)>0$ (use $u(y)>u(\sigma))$ the two sides,
 we obtain
$$\frac{d}{dy}u^{2}(y) \leq -2\xi_{1}^{2} \frac{u(y)}{c(y)}\int_{y}^{H}u(\sigma)c(\sigma)\,d\sigma\leq -2\frac{\xi_{1}^{2}}{c(y)}\int_{y}^{H}u^{2}(\sigma)c(\sigma)\,d\sigma, \, y\in\lbrack \beta,H\rbrack.$$

 Integration over $\lbrack \beta,H\rbrack,$ 
  yields
\begin{equation*}
\label{equation-lowerbound:4}
0- u^{2}(\beta) \leq  -2\xi_{1}^{2}\int_{\beta}^{H}\frac{1}{c(y)}\left( \int_{y}^{H}u^{2}(\sigma)c(\sigma)\,d\sigma\right)dy = -2\xi_{1}^{2}\int_{\beta}^{H}\left( \int_{\beta}^{\sigma}u^{2}(\sigma)\frac{c(\sigma)}{c(y)}\,dy\right)d\sigma,
\end{equation*}
hence
\begin{equation}
\label{equation-div-lowerbound:2}
u^{2}(\beta)\geq 2\,\frac{c_{m}}{c_{M}}\,\xi_{1}^{2}\int_{\beta}^{H}(\sigma-\beta)u^{2}(\sigma)\,d\sigma.
\end{equation}
The two inequalities ~\eqref{equation-div-upperbound:1} and ~\eqref{equation-div-lowerbound:2} give
\begin{equation}
\label{equation-div-upperlowerbound:1}
(a-\beta)\int_{a}^{b}u^{2}(\sigma)\,d\sigma\leq\int_{\beta}^{H}(\sigma-\beta)u^{2}(\sigma)\,d\sigma\leq \gamma\, \frac{(1 + \xi_{M}^{2})^{s}}{\xi_{1}^{2}}\Vert u\Vert^{2}_{L^{2}(0,H)}.
\end{equation}
 Taking any $\frac12<s<1$ concludes the proof of ~\eqref{equation-decreasing:1}  since $\mu_{k}^{2}\geq \frac{\xi_{M}^{2}}{c_M}\geq \xi_{1}^{2}\geq \frac{\varepsilon}{c_{1}}\mu_{k}^{2}.$

The same reasoning works if $\lbrack a,b\rbrack\subseteq (0,\alpha)$ since in this case $(cu')(0)>0$ hence $u$ is increasing and positive  in the interval.
\end{proof}

%%%%%%%%%%%%%%%%%%%%%%%%%%%%%%%%%%%%%%%%%%%%%%%%%%%%%%%%%%%%%%%%%
%%%%%%%%%%%%%%%%%%%%%%%%%%%%%%%%%%%%%%%%%%%%%%%%%%%%%%%%%%%%%%%%
\section{\textbf{NON-GUIDED WAVES- THE REGULAR CASE }}
\label{section-regular-ng}
%%%%%%%%%%%%%%%%%%%%%%%%%%%%%%%%%%%%%%%%%%%%%%%%%%%%%%%%%%%%%%%%%
%%%%%%%%%%%%%%%%%%%%%%%%%%%%%%%%%%%%%%%%%%%%%%%%%%%%%%%%%%%%%%%%
   An (infinite) set of non-guided normalized eigenfunctions  is characterized
  by the fact that in each layer $\Omega_{a,b}$ there is a uniform positive lower bound for the masses in the layer, valid for all  elements of the set.

  As observed in the Introduction, for each eigenvalue $\lambda$ of $A,$ there exists at least one pair $(k,\ell)$ so that $\lambda=\beta_{k,\ell}$ is the $\ell$-th eigenvalue of $A_{k}.$

In this section we prove Theorem ~\ref{theo-C1-noguided} under the following regularity assumption on $c(y).$

\begin{assume}\label{assume-c(y)}
We assume that $c(y)\in C^2[0,H].$
\end{assume}

    We consider the equation 
   \begin{equation}
\label{regular-introduction:1-ng}
\aligned
 (c(y)u(\ylambk)')' + (\lambda-c(y)\mu^{2}_{k})u(\ylambk)\\
 =(c(y)u(\ylambk)')' + \mu_{k}^{2}\, p(\ylambk)u(\ylambk)=0,\\
 p(\ylambk)= \frac{\lambda}{\mu^2_{k}}-c(y).
 \endaligned
\end{equation}
$$u(0;\lambda,k) = u(H;\lambda,k) = 0.$$
 We shall deal in this section with eigenvalues $\lambda$ such that
 (see ~\eqref{eqAceps})
$$(\mu_k,\lambda)\in  \mathscr{A}^{c}_\eps.$$
In view of ~\eqref{eqAceps}

\begin{equation}
\label{ineq-regular-intermedzone:2}
p(\ylambk) \geq \eps, \quad y\in \lbrack 0, H\rbrack.
\end{equation}
In addition, for some $\Lambda>\eps$ we limit the eigenvalues by $\lambda<(c_M+\Lambda)\mu_k^2.$
\begin{thm}
\label{thm-ng-Liouville:1-new}Let
$$\mathscr{A}^{c}_{\eps,\Lambda}= \mathscr{A}^{c}_\eps\cap \set{\lambda\leq(c_{M} + \Lambda)\mu_{k}^{2}}.$$

 Let $u(\ylambk)$ be a normalized solution to \eqref{regular-introduction:1-ng}. Then for every interval $(a,b)\subseteq (0,H)$
\begin{equation}
\label{ineq-ng-Liouville:1-new}
\inf_{(\mu_{k},\lambda)\in \mathscr{A}^{c}_{\eps,\Lambda}}\int_{a}^{b} u(\ylambk)^{2}dy>0.
\end{equation}

\end{thm}
\begin{proof}
Define a new coordinate $t=t(y)$ so that
\be\label{eqywt-regular} \frac{dy}{dt}=c(y),\quad t(0)=0.\ee
Let $$w(t;\lambda,k)=u(\ylambk),\quad t\in [0,\widetilde{H}],\,\,\widetilde{H}=\int_0^Hc^{-1}(y)dy .$$ Then $w(t;\lambda,k)<0$ in a neighborhood of $t(0)$ ($t(0)$ excluded) while by ~\eqref{regular-introduction:1-ng} $\frac{dw}{dt}$ is  increasing (as long as $u<0$). It follows that there is a first point $t_1=t(z_1)$
 so that $w(t_1;\lambda,k)=u(z_1;\lambda,k)=0.$
In a similar way we arrive at a second zero $w(t_2;\lambda,k)=u(z_2;\lambda,k)=0$ and so forth.

In terms of the new coordinate $t$ the equation ~\eqref{regular-introduction:1-ng} can be rewritten as
\begin{equation}
\label{equation-dt2wt-regular}\aligned
 \frac{d^2}{dt^2}w(t;\lambda,k) + \mu_{k}^{2}\, \widetilde{p}(t;\lambda,k)w(t;\lambda,k)=0,\\
\widetilde{p}(t;\lambda,k)= \frac{\lambda c(y(t))}{\mu^2_{k}}-c^2(y(t)),\quad t\in [0,\widetilde{H}].
\endaligned \end{equation}
 Instead of ~\eqref{ineq-regular-intermedzone:2} we now have, since $\lambda>(c_{M}+\eps)\mu^2_k,$
 \begin{equation}
\label{ineq-regular-intermedzone:3}
\widetilde{p}(t;\lambda,k) \geq \kappa_1=\eps c_m, \quad t\in [0,\widetilde{H}],
\end{equation}
and also, since $\lambda<(c_M+\Lambda)\mu^2_k,$
\begin{equation}
\label{ineq-regular-intermedzone:4}
\widetilde{p}(t;\lambda,k) \leq \kappa_2=(c_M+\Lambda)c_M-c_m^2, \quad t\in [0,\widetilde{H}].
\end{equation}
 We apply the Liouville transformation ~\cite[Chapter IV]{erdelyi} to Equation ~\eqref{equation-dt2wt-regular}:

       \be\label{eqliouv}
       \xi=\int_0^t\sqrt{\widetilde{p}(s;\lambda,k)}ds,\quad \eta(\xi;\lambda,k)=[\widetilde{p}(t;\lambda,k)]^{\frac14}w(t;\lambda,k).
       \ee
     Note that $\xi\in[0,\overline{H}],\,\,\overline{H}=\int_0^{\widetilde{H}} \sqrt{\widetilde{p}(s;\lambda,k)}ds\leq \widetilde{H} \sqrt{\kappa_2}.$

     The function $\eta(\xi;\lambda,k)$ satisfies the equation
     \be\label{eqetaxi}
     \frac{d^2\eta}{d\xi^2}+\mu_k^2\eta=\rho(\xi;\lambda,k)\eta,
     \ee
       where
       $$\rho(\xi;\lambda,k)=\frac14\frac{\widetilde{p}''(t;\lambda,k)}{\widetilde{p}(t;\lambda,k)^2}-
       \frac{5}{16}\frac{\widetilde{p}'(t;\lambda,k)^2}{\widetilde{p}(t;\lambda,k)^3}.$$

    Note that the form of ~\eqref{eqetaxi} is the starting point for the asymptotic behavior of solutions involving potential perturbations. However in our case  the potential depends on the spectral parameter.

        Observe that under our assumptions
         the family
          $$\mathscr{B}=\set{\rho(\xi;\lambda,k),\quad (\mu_k,\lambda)\in\mathscr{A}^{c}_{\eps,\Lambda}}$$
           is \textbf{uniformly bounded}.

       Since $\eta(0;\lambda,k)=0,$ Equation ~\eqref{eqetaxi}
       entails, with $\alpha=\mu_k^{-1}\eta'(0;\lambda,k),$
       \be\label{eqetaxiinteg}
       \eta(\xi;\lambda,k)=\alpha \sin(\mu_k\xi)+
       \mu_k^{-1}\int_0^\xi\sin(\mu_k(\xi-\tau))
       \rho(\tau;\lambda,k)
       \eta(\tau;\lambda,k)d\tau,\quad
       \xi\in[0,\overline{H}].
       \ee

       The uniform boundedness of the family $\mathscr{B}$ implies that
              the Volterra integral equation
       ~\eqref{eqetaxiinteg} is
       solvable for any sufficiently large $\mu_k$ ~\cite[Chapter 1]{widom}. Furthermore,  there exist $\gamma,\,\mu_{0}>0$ so that
         \be\label{eqcompareetasin}
         |\eta(\xi;\lambda,k)-\alpha \sin(\mu_k\xi)|\leq \gamma\mu_k^{-1},\quad \xi\in[0,\overline{H}],\,\,\mu_k>\mu_0,\,\rho\in \mathscr{B}.
         \ee
         We now make the following observations.
         \begin{itemize}
         \item Recall that $\int_0^H u(\ylambk)^2\,dy=1,$ hence in light of ~\eqref{ineq-regular-intermedzone:3}-~\eqref{ineq-regular-intermedzone:4} and $\lambda\leq (c_{M} + \Lambda)\mu_{k}^{2}$
          there exist two constants $0<\zeta_1<\zeta_2<\infty$ so that
             \be\label{eqeta2integ}
            \zeta_1\leq \int_0^{\overline{H}}\eta(\xi;\lambda,k)^2d\xi\leq \zeta_2,\quad \mu_{k} >\mu_0.
             \ee
             \item It follows from ~\eqref{eqcompareetasin}-\eqref{eqeta2integ} that there exist two constants $0<r_1<r_2<\infty$ so that
             \be\label{eqalphar1r2}
                r_1\leq |\alpha|\leq r_2, \quad \mu_{k}>\mu_0.
             \ee
         \end{itemize}
          Let $(\xi_1,\xi_2)$ be the interval corresponding to $(a,b).$ The Cauchy-Schwarz inequality, applied to 
         ~\eqref{eqcompareetasin}, yields
         $$
          \| \eta(\xi;\lambda,k)\|_{L^2((\xi_1,\xi_2),d\xi)}\geq |\alpha|\| \sin(\mu_k\xi)\|_{L^2((\xi_1,\xi_2),d\xi)}-\gamma\mu_k^{-1}(\xi_2-\xi_1)^{\frac12},
         $$
         namely,
         \be\label{eqetasin}
         \| \eta(\xi;\lambda,k)\|_{L^2((\xi_1,\xi_2),d\xi)}\geq |\alpha|\Big[\frac{\xi_2-\xi_1}{2}-\frac{1}{2\mu_k}\sin(2\mu_k\xi)\Big|^{\xi_2}_{\xi_1}\Big]^{\frac12}-
         \gamma\mu_k^{-1}(\xi_2-\xi_1)^{\frac12}.
         \ee
         Increasing $\mu_0$ (if needed) so that $\mu_0>\frac{4}{\xi_2-\xi_1}$ we conclude that
          $$\| \eta(\xi;\lambda,k)\|_{L^2((\xi_1,\xi_2),d\xi)}\geq ( \frac14|\alpha|- \gamma\mu_k^{-1})(\xi_2-\xi_1)^{\frac12} .$$
         
                         Noting ~\eqref{eqalphar1r2} and requiring $\mu_0>\frac{8\gamma}{r_1},$ we finally obtain
             $$ \int_{\xi_1}^{\xi_2}\eta(\xi;\lambda,k)^2d\xi\geq\Big( \frac{r_1}{8}\Big)^2(\xi_2-\xi_1),\quad \mu_{k}>\mu_0.$$
             
             Note that there are at most finitely many normalized eigenfunctions associated
              with values $\mu_k<\mu_0,$ since $\lambda$ is bounded from above.

              Switching back to the original variable $y$ and the function $u(\ylambk)$  we get ~\eqref{ineq-ng-Liouville:1-new}.
  \end{proof}   

 \section{\textbf{NON-GUIDED WAVES-BEYOND THE REGULAR CASE }}
\label{section-sufficient condition-ng}
$\empty$\\
\noindent
  
  In this  section we prove Theorem  ~\ref{theo-C1-noguided} for a coefficient $c(y)$ that does not satisfy the regularity Assumption ~\ref{assume-c(y)}. We first establish some properties related to the oscillatory character of the solutions.

 As observed in the Introduction, for each eigenvalue $\lambda$ of $A,$ there exists at least one pair $(k,\ell)$ so that $\lambda=\beta_{k,\ell}$ is the $\ell$-th eigenvalue of $A_{k}.$
  Let $\lambda = \beta_{k,\ell}>0$ be an eigenvalue of $-\nabla\cdot c(y)\nabla$ in $L^2(\Omega,dx'dy)$ and $u(y;\lambda,k):=u_{\lambda, k}(y)$ the normalized  associated (reduced) eigenfunction (as in \eqref{equation-introduction:1} and \eqref{equation-introduction:1bis}). It satisfies the equation

\begin{equation}
\label{equation-introduction:1-ng}
\aligned
 (c(y)u(\ylambk)')' + (\lambda-c(y)\mu^{2}_{k})u(\ylambk)\\
 =(c(y)u(\ylambk)')' + \mu_{k}^{2}\, p(\ylambk)u(\ylambk)=0,\\
 p(\ylambk)= \frac{\lambda}{\mu^2_{k}}-c(y).
 \endaligned
\end{equation}
$$u(0;\lambda,k) = u(H;\lambda,k) = 0.$$
 We shall deal in this section with eigenvalues $\lambda$ such that
 (see ~\eqref{eqAceps})
$$(\mu_k,\lambda)\in  \mathscr{A}^{c}_\eps.$$
In view of ~\eqref{eqAceps}

\begin{equation}
\label{ineq-intermedzone:2}
p(\ylambk) \geq \eps, \quad y\in \lbrack 0, H\rbrack.
\end{equation}

 The aim of the subsequent subsections is to claim that the set of eigenfunctions
associated with eigenvalues satisfying $(\mu_k,\lambda)\in\mathscr{A}^{c}_\eps ,$
for any  $c(y)\in\mathscr{K}$ (see ~\eqref{eq-def-K})
consists of non-guided eigenfunctions when a particular sufficient condition is satisfied, with $c(y)$ less regular.

%%%%%%%%%%%%%%%%%%%%%%%%%%%%%%%%%%%%%%%%%%%%%%%%%%%%%%%%%%%%%%%%%
%%%%%%%%%%%%%%%%%%%%%%%%%%%%%%%%%%%%%%%%%%%%%%%%%%%%%%%%%%%%%%%%%%

\noindent\\

Consider a pair $(\mu_k,\lambda)\in\mathscr{A}^{c}_\eps .$
Let $u(\ylambk)$ be a normalized solution  to ~\eqref{equation-introduction:1-ng}, associated with $(\mu_{k}, \lambda).$
 Without loss of generality we may assume that $u'(0;\lambda,k)<0.$
 
 As in  ~\eqref{eqywt-regular}  we define a new coordinate $t$ by
\be\label{eqnewyt}\frac{dy}{dt}=c(y),\quad t(0)=0,\,t(H)=\widetilde{H}=\int_0^Hc^{-1}(y)dy.\ee
 However, because of the irregularity of $c(y)$ we need to pay attention to domain considerations. 
  The two-sided boundedness of $c(y)$ readily implies that $y(t),\,\,t(y)$ are Lipschitz on $[0,\widetilde{H}],\,\,[0,H],$ respectively.  In particular, they are in $H^1$ on their respective domains. Now $u(\ylambk)$ is bounded (indeed $H\ddot{o}lder $ continuous in $[0,H]$) since it is in $ H^1_0.$ Furthermore,  $c(y)u'(\ylambk)\in H^1$ hence it is uniformly bounded, and thus $u'(\ylambk) $ is uniformly bounded. We infer that $u(\ylambk)$ is Lipschitz, hence $w(t;\lambda,k)=u(y(t);\lambda,k)$ is Lipschitz and
  $$\frac{d}{dt}w(t;\lambda,k)=c(y(t))u'(\ylambk)|_{y=y(t)}.$$
  From equation ~\eqref{equation-introduction:1-ng} we infer that $(c(y)u'(\ylambk))'$ is bounded so $c(y)u'(\ylambk)$ is Lipschitz hence so is the composition  $c(y(t))u'(\ylambk)|_{y=y(t)}.$ It follows that
  $$\frac{d^2}{dt^2}w(t;\lambda,k)=(c(y)u'(\ylambk))'|_{y=y(t)}\cdot c(y(t)).$$
  
We may switch to the new coordinate $t$ and rewrite  the equation ~\eqref{equation-introduction:1-ng}  as
\begin{equation}\aligned
\label{equation-dt2wt}
 c(y(t))^{-1}\frac{d^2}{dt^2}w(t;\lambda,k) + \mu_{k}^{2}\, p(t;\lambda,k)w(t;\lambda,k)=0,\\
 p(t;\lambda,k)= \frac{\lambda}{\mu^2_{k}}-c(y(t)).
\endaligned \end{equation}

Let
\begin{equation*}
Z(\lambda,k) = \{z_{0} =0<z_{1}<z_{2}<\ldots < z_{ \mathfrak{s}} = H,\,\,\, u(z_{i};\lambda, k) = 0, \quad 0\leq  i\leq \mathfrak{s}\}\subseteq \lbrack 0, H\rbrack
\end{equation*}
be the set of zeros of the function $u(\ylambk).$

 We have $w(t;\lambda,k)<0$ in some interval $(t(0),t(0)+\delta)$ while by ~\eqref{equation-dt2wt} $\frac{dw}{dt}$ is  increasing (as long as $u<0$). It follows that there is a first point $t_1=t(z_1)$
 so that $w(t_1;\lambda,k)=u(z_1;\lambda,k)=0.$
In a similar way we arrive at a second zero $w(t_2;\lambda,k)=u(z_2;\lambda,k)=0$ and so forth.

 It follows from \eqref{ineq-intermedzone:2} and the comparison principle ~\cite[Section X.6]{birkhoff},  ~\cite[Section 8.1]{coddington} applied to ~\eqref{equation-dt2wt} that there exists a constant $\gamma>0,$ independent of $\eps,\,\lambda,\,k$ such that

$$t_{i+1} -t_{i}  \leq \gamma\mu_k^{-1}\sqrt{ \frac{1}{\eps}}, \quad 0\leq i\leq \mathfrak{s} - 1.$$
In view of ~\eqref{eqnewyt} this estimate can be restated in terms of the original coordinate $y,$ perhaps with a different constant $\gamma>0,$
\begin{equation}
\label{ineq-intermedzone:3}
z_{i+1} -z_{i}  \leq \gamma\mu_k^{-1}\sqrt{ \frac{1}{\eps}}, \quad 0\leq i\leq \mathfrak{s} - 1.
\end{equation}

  The following claim extends ~\eqref{ineq-intermedzone:3} and will be useful in the sequel. It says that the distance between two consecutive zeros of an (oscillatory) eigenfunction can be made arbitrarily small, if we drop a finite number of eigenfunctions associated with ``low'' eigenvalues. 
\begin{claim}
\label{claim:1}
Let $c(y)\in\mathscr{K}.$ For each $\alpha>0,\,\eps>0$ there exists $\lambda_{\alpha,\varepsilon}$ such that $\lambda>\lambda_{\alpha,\varepsilon}$
 implies
 $$z_{i+1} -z_{i} < \alpha,\quad 0\leq i\leq \mathfrak{s} - 1.$$

 In particular, $\lambda_{\alpha,\varepsilon}$ can be chosen uniformly for all $c(y)\in\mathscr{K}.$ For each $c(y)\in\mathscr{K}$ there are at most finitely many eigenvalues $\lambda<\lambda_{\alpha,\varepsilon}.$
\end{claim}
\begin{proof}
Recall that we are assuming $(\mu_k,\lambda) \in\mathscr{A}^{c}_\eps$ so that by ~\eqref{ineq-intermedzone:2}
$p(y(t);\lambda,k)\geq \eps.$ Thus by the comparison principle it suffices to compare ~\eqref{equation-dt2wt} with the constant coefficient equation 
 $$\frac{d^2}{dt^2}w(t;\lambda,k) + \eta^2w(t;\lambda,k)=0.$$
 For any  sufficiently large $\eta$  the distance between two consecutive zeros is less than $\alpha/2.$
  Pick $\mu_{0,\varepsilon}> \frac{\eta}{\sqrt{\eps c_m}}.$  Then  $c(y)p(\ylambk)\mu_k^2>\eta^2$ if $\mu_k>\mu_{0,\varepsilon}.$
 Next choose $\lambda_{0,\varepsilon} = (c_M +\eps)\mu_{0,\eps}^2$.
 Clearly for any $\lambda>\lambda_{0,\varepsilon}$ and $\mu_k\leq \mu_{0,\varepsilon}$ we have $c(y)p(\ylambk)\mu_k^2>\eta^2.$
 Note that there are at most finitely many pairs $(\mu_k,\lambda)\in \mathscr{A}^{c}_\eps$ with $\mu_k\leq\mu_{0,\varepsilon}$ and $\lambda\leq \lambda_{0,\varepsilon}.$
   Finally, take 
   $\lambda_{\alpha,\varepsilon}=\max \Big[(c_M+\eps)\mu_{0,\varepsilon}^2,\lambda_{0,\varepsilon}\Big].$

   \end{proof}

\noindent
In particular, without further assumptions, the solutions are oscillatory  between consecutive zeros. However, in various sub-intervals their amplitudes might decay to zero, hence concentrating in the complementary domain. It is precisely this behavior that we seek to exclude.

 The implications of the assumption that $c(y)$ is subject only to the {\textbf{\textit{minimal amplitude hypothesis}}}  ( Definition ~\ref{definition-hypo-minam:1}) will now be studied. No regularity is required of $c(y),$
and only condition $\mbox{\bf(H)}$ (see ~\eqref{eqassumptioncy}) is imposed.\\
We have already seen that the lack of regularity does not affect the oscillatory character of the solutions. The remaining issue is to see that the masses of the oscillatory solutions  $\set{ u(\ylambk),\,(\mu_{k},\lambda)\in\mathscr{A}^{c}_\eps}$  in any interval remain uniformly bounded away from zero.\\ This is addressed in the following theorem
which is a somewhat more detailed form of  Theorem  ~\ref{theo-C1-noguided}. Its proof is straightforward, reducing the non-concentration issue to a study of the minimal amplitude hypothesis for various functional classes.

\begin{thm}
 \label{thmcyirreg-minamp} Let $c(y)\in \mathscr{K}$ satisfy the minimal amplitude hypothesis (Definition ~\ref{definition-hypo-minam:1}).
 Consider the family $\set{ u(\ylambk),\,(\mu_{k},\lambda)\in\mathscr{A}^{c}_\eps}$ of  normalized solutions to ~\eqref{equation-introduction:1-ng}.

                Then, for every interval $(a,b)\subseteq (0,H),$ there exist  constants
                \begin{itemize}
                \item $d>0$ depending on $\eps,\,c_M,\,c_m,\, b-a,\,\mathfrak{r_{c,\eps}},$
                    \item $\lambda_0>0$ depending on $\eps,\,c_M,\,c_m,\, b-a$
                \end{itemize}
                              such that
                \be\label{equyneq0}
                \int_{a}^{b}\, u(y;\lambda,k)^2dy\geq d,\quad \,\,
                 \lambda>\lambda_0.
                \ee
                 This estimate is equivalent to  $\int_{\Omega_{a,b}}\,v_{\lambda}(x)^2dx\geq d,$   where $v_\lambda$ is the eigenfunction of $-\nabla\cdot c(y)\nabla$ associated with $\lambda$ and $u(y;\lambda,k)$ (see ~\eqref{equation-introduction:1bis}).
             \end{thm}

             \begin{proof}

 Note that the function $u(\ylambk)$ satisfies the equation
\be\label{equeigen}
[c(y)u'(\ylambk)]'=-(\lambda-c(y)\mu_k^2)u(\ylambk), \quad \lambda>(c_M+\eps)\mu_k^2.
\ee
Assume that 
    $$u(\ylambk)<0,\quad y\in (z_i,z_{i+1}).$$

We infer that $cu'$ is an increasing function in $(z_i,z_{i+1}).$ We change to a new coordinate $t$  as in ~\eqref{eqnewyt}, namely,
$$ \frac{dy}{dt}=c(y),\quad t(z_i)=0,\,\,t(z_{i+1})=t_{1}.$$
Let $w(t;\lambda,k)=u(\ylambk).$ Then $\frac{dw}{dt}$ is an increasing function in $(0,t_{1}).$
It follows that $w(t;\lambda,k)$ is convex in the interval. Let $z_{i+\half}$ be the (unique) zero of $u'(\ylambk)$ in the interval and denote $t_{\half}=t(z_{i+\half})$
\begin{equation}
\label{DF-ineq-intermedzone:5}
w(t_{\half};\lambda,k)^2=\max\limits_{0\leq t\leq t_1}w(t;\lambda,k)^2\geq \mathfrak{r}_{c,\eps}^2
\end{equation}
and as in the proof of ~\cite[Theorem 4.5]{BBD:5}
\be\label{eqwrc}
\int_0^{t_1}w(t;\lambda,k)^2dt\geq \frac13\mathfrak{r}_{c,\eps}^2(t_1-0),
\ee
hence
$$\int_{z_{i}}^{z_{i+1}}u(y;\lambda,k)^2dy\geq \frac{c_m}{3}\mathfrak{r}_{c,\eps}^2(t_1-0)\geq\frac{
c_m}{3c_M}\mathfrak{r}_{c,\eps}^2(z_{i+1}-z_i).$$

 This concludes the proof of the theorem.
 \end{proof}

 %%%%%%%%%%%%%%%%%%%%%%%%%%%%%%%%%%%%%%%%%%%%%%%%%%%%%%%%%%%%%%%%%%
%%%%%%%%%%%%%%%%%%%%%%%%%%%%%%%%%%%%%%%%%%%%%%%%%%%%%%
\subsection{\textbf{MORE ON THE MINIMAL AMPLITUDE HYPOTHESIS}}\label{subsec-more-minim}$\empty$

Recall that the minimal amplitude was defined by ~\eqref{equation-definition-hypo-minam:1}:

\be\label{eqdefrceps} \mathfrak{r}_{c,\eps}^{2} = \inf_{\begin{array}{c} y\in \lbrack 0, H\rbrack,\\ (\mu_k,\lambda)\in\mathscr{A}^{c}_\eps\end{array}}\lbrack u_{\lambda,k}(y)^{2} + {(c(y)u'_{\lambda,k}(y))^{2}}\rbrack.\ee
We now consider this quantity in more detail.

The next proposition supplements Equation \eqref{DF-ineq-intermedzone:5}.
\begin{prop}
\label{prop-minam:0}
  Let $u(y;\lambda,k)$ be a normalized solution to \eqref{equation-introduction:1-ng} where $c(y)\in \mathscr{K}$ and $(\mu_{k},\lambda)\in\mathscr{A}^{c}_\eps.$ Let
\begin{equation}
\label{definition-minimal-amplitude:2}
\mathfrak{r}_{c,{\color{red}\eps},\mu_{k},\lambda}^{2}= \inf_{y\in \lbrack 0, H\rbrack}\lbrack u(y;\lambda,k)^{2} + (c(y)u'(y;\lambda,k))^{2}\rbrack.
\end{equation}
Then there exists a positive
$\tilde{\lambda}_0$ such that for any  $\lambda>\tilde{\lambda}_0,$
\begin{equation}
\label{definition-minimal-amplitude:3}
\mathfrak{r}_{c,\eps,\mu_{k},\lambda}^{2} =\min_{0\leq i < \mathfrak{s}} u(z_{i + \frac{1}{2}};\lambda,k)^{2}.
\end{equation}
\end{prop}

\begin{proof} We introduce again the change of variable  \eqref{eqnewyt} and suppose that $w(t;\lambda,k)=u(y(t);\lambda,k)$ is convex in the interval $(t_{i}, t_{i+1})$ as in the proof of Theorem ~\ref{thmcyirreg-minamp}.
 In light of Equation  ~\eqref{equation-dt2wt} the function $w(t;\lambda,k)$ satisfies
\begin{equation}
\label{equation-minimal-amplitude:1}
w''(t;\lambda,k) +c(y(t))\mu^{2}_{k}p(t;\lambda,k) w(t;\lambda,k) = 0,
\end{equation}
hence
\begin{equation}
\label{equation-minimal-amplitude:2}
\frac{d}{dt}\lbrack w(t;\lambda,k)^{2} + w'(t;\lambda,k)^{2}\rbrack = 2 \big(1-c(y(t))\mu^{2}_{k}p(t;\lambda,k)\big) w'(t;\lambda,k)w(t;\lambda,k).
\end{equation}
In the interval $(t_i,t_{i+1})$ we have $w(t;\lambda,k)<0$ and by convexity, with $z_{i + \frac{1}{2}} =y(t_{i+\frac{1}{2}}),$
\begin{equation}
\label{equation-minimal-amplitude:3}
w'(t;\lambda,k)\left\lbrace\begin{array}{l}
<0, t\in \lbrack t_{i}, t_{i+\frac{1}{2}}),\\
>0, t\in ( t_{i+\frac{1}{2}}, t_{i+1}\rbrack.
\end{array}\right.
\end{equation}
               As in the proof of  Claim ~\ref{claim:1}, we can now find  $\tilde{\lambda}_0$ so that $c(y)\mu^{2}_{k}p(y;\lambda,k)>1$ for $\lambda> \tilde{\lambda}_0.$
               Inserting this in ~\eqref{equation-minimal-amplitude:2} we obtain

\begin{equation*}
\frac{d}{dt}\Big[ w(t;\lambda,k)^{2} + w'(t;\lambda,k)^{2}\Big]\left\lbrace\begin{array}{l}
<0, t\in ( t_{i}, t_{i+\frac{1}{2}}),\\
>0, t\in ( t_{i+\frac{1}{2}}, t_{i+1}).
\end{array}\right.
\end{equation*}
As $u(y;\lambda,k)^{2} + (c(y)u'(y;\lambda,k))^{2} = w(t;\lambda,k)^2 + w'(t;\lambda,k)^2 $, Equation \eqref{definition-minimal-amplitude:3} clearly follows by considering all intervals. 
\end{proof}

%%%%%%%%%%%%%%%%%%%%%%%%%%%%%%%%%%%%%
As a corollary to the proof of Proposition ~\ref{prop-minam:0} we have
\begin{cor}
\label{cor-ng-minam-hyp:1}
If  $\lambda>\tilde{\lambda}_0$ then
for every $0\leq i < \mathfrak{s}$
\begin{equation}
\min \lbrack \vert c(z_{i}) u'(z_{i};\lambda,k)\vert^{2}, \vert c(z_{i+1})u'(z_{i+1};\lambda,k)\vert^{2}\rbrack\geq \vert u(z_{i+\frac{1}{2}};\lambda,k)\vert^{2}.
\end{equation}
\end{cor}

%%%%%%%%%%%%%%%%%%%%%%%%%%%%%%%%%%%%%%%%%%%%%%%%%%%%%%%%%%%%%%%%%%

%%%%%%%%%%%%%%%%%%%%%%%%%%%%%%%%%%%%%%%%%%%%%%%%%%%%%%%%%%%%%%%%%%
\subsection {\textbf{$c(y)$ LIPSCHITZ}} $\empty$\label{subsublip}

   The first case of coefficients to be considered in the following proposition is that of Lipschitz functions.

\begin{prop}
\label{prop-minam:1bis}
Let $c(y)\in \mathscr{K}$ be in $W^{1,\infty}(\lbrack 0, H\rbrack).$ Then it satisfies the minimal amplitude hypothesis  with respect to $\mathscr{A}^{c}_\eps.$
\end{prop}

%%%%%%%%%%%%%%%%%%%%%%%%%%%%%%%%%%%%%%%%%%%%%%%%%%%%%%%%%%%%%%%%%%
\begin{proof}
Let $u(\ylambk)$ be a normalized solution to ~\eqref{equation-introduction:1-ng}.
 We just need to prove the estimate ~\eqref{eqrcgreat0}.  We apply the variable change ~\eqref{eqnewyt}. Equation \eqref{equation-introduction:1-ng} can be rewritten as in ~\eqref{equation-minimal-amplitude:1}
 \begin{equation}
\label{eq-intermedzone:4}
w''(t;\lambda,k) +  \alpha^2(t;\lambda,k)w(t;\lambda,k)=0,
\end{equation}
where
 \begin{equation}
\label{eq-intermedzone:5}
\alpha(t;\lambda,k) =\mu_k\sqrt{c(y(t)) p(y(t))} = \mu_k\sqrt{c(y(t)) \frac{\lambda}{\mu^2_{k}}-c^2(y(t))}.
\end{equation}
Observe that in light of ~\eqref{ineq-intermedzone:2}
 \begin{equation}
\label{ineq-intermedzone:10}
\sqrt{\eps c_m}\mu_{k} \leq\alpha(t;\lambda,k).
\end{equation}
\noindent
We now replace $w,\,w'$ by $w_1,\,w_2$ as follows (suggested in the  recent paper ~\cite{ABM:1}).
\begin{equation*}
w_{1}(t;\lambda,k) = w(t;\lambda,k), \quad w_{2}(t;\lambda,k) = \frac{w'(t;\lambda,k)}{\alpha(t;\lambda,k)},
\end{equation*}
and note that the vector function $U(t;\lambda,k)  = \left(\begin{array}{c}
w_{1}(t;\lambda,k) \\
w_{2}(t;\lambda,k)
\end{array}\right)$
 satisfies
\begin{equation}
\label{eq-intermedzone:6}
U'(t;\lambda,k) = \left(\begin{array}{cc}
0 & \alpha(t;\lambda,k)\\
-\alpha(t;\lambda,k) & 0
\end{array}\right)U(t;\lambda,k)  + \left(\begin{array}{cc}0&0\\0&-\frac {\alpha'(t;\lambda,k)}{\alpha (t;\lambda,k)}
 \end{array}\right)U(t;\lambda,k).
\end{equation}
It readily follows that
\begin{equation}
\label{eq-intermedzone:6a}
\frac{d}{dt}(w_{1}^{2} + w_{2}^{2}) = - 2\frac {\alpha't;\lambda,k)}{\alpha (t;\lambda,k)} w_{2}^{2}.
\end{equation}

  Now
  $$\frac{\alpha'}{\alpha}=\frac{1}{2cp}(cp)',$$
  hence
  
  $$\frac{\alpha'}{\alpha}=
  \frac12\, \left(1-\frac{c(y(t))\mu_k^2}{\lambda -c(y(t))\mu_k^2}\right)\frac{c'(y(t))}{c(y(t))}.$$
  Since $\mu_k^2c(y)\leq \mu_k^2c_M$ and $\lambda -c(y(t))\mu_k^2\geq \eps \mu_k^2$ 
   it follows that
  \begin{equation*}\label{eqalpha'alpha}
  \Big|\frac{\alpha'}{\alpha}\Big|\leq \frac{1}{2}\left(1 + \frac{c_M}{\eps}\right)\,\sup\limits_{y\in [0,H]}\,\frac{|c'(y)|}{c(y)},
  \end{equation*}

  that implies
  \begin{equation}
\label{ineq-intermedzone:11}
-C (w_{1}^{2} + w_{2}^{2})\leq \frac{d}{dy}(w_{1}^{2} + w_{2}^{2})\leq C (w_{1}^{2} + w_{2}^{2}),
\quad C=\left(1 + \frac{c_M}{\eps}\right)\,\sup\limits_{y\in [0,H]}\,\frac{|c'(y)|}{c(y)}.
\end{equation}

 Since $w_{1}^{2} + w_{2}^{2}>0$ in the interval $\lbrack 0, \tilde{H}\rbrack$ we conclude that there exists a constant $R>0,$ so that for all $(\mu_k,\lambda)\in \mathscr{A}^{c}_\eps$ and  for any $t_{1},t_{2}\in \lbrack 0, \tilde{H}\rbrack$
  \begin{equation}
\label{ineq-intermedzone:12}
(w_{1}^{2} + w_{2}^{2})(t_{2})\leq R (w_{1}^{2} + w_{2}^{2})(t_{1}).
\end{equation}
Furthermore, from the normalization of $u(\ylambk)$, i.e. $1 = \int_{0}^{H}u(\ylambk)^{2}dy = \int_0^{\tilde{H}} w_1^2(t)c(y(t))\,dt,$ we have $1\leq c_M \int_0^{\tilde{H}} (w_1^2 + w_2^2)(t) dt\leq c_M R \tilde{H}(w_1^2 + w_2^2)(t_1), \forall t_1\in (0,\tilde{H}).$ It follows that there exists a constant $\eta>0,$ so that for all $(\mu_k,\lambda)\in \mathscr{A}^{c}_\eps,$

  \begin{equation}
\label{ineq-intermedzone:13}
(w_{1}^{2} + w_{2}^{2})(t)\geq \eta, \quad t\in \lbrack 0, \tilde{H}\rbrack.
\end{equation}
This estimate, combined with the definition of $w_{1}, w_{2}$ and \eqref{ineq-intermedzone:10} implies the required estimate \eqref{eqrcgreat0} and concludes the proof of the proposition.
\end{proof}
\noindent
%%%%%%%%%%%%%%%%%%%%%%%%%%%%%%%%%%%%%%%%%%%%%%%%
\noindent

 %%%%%%%%%%%%%%%%%%%%%%%%%%%%%%%%%%%%%%%%%%%%%%%%%%%%%%%%%%%%%%%%
%%%%%%%%%%%%%%%%%%%%%%%%%%%%%%%%%
\subsection{\textbf{ $c(y)$ PIECEWISE CONSTANT}}\label{subsecpiececonst}
$$ $$

 We turn next to the case that $c(y)$ is a piecewise constant function. In  Theorem ~\ref{prop-minam:1-new} below we discuss our most general case,  namely  $c(y)$ of bounded variation. To this end, a detailed treatment of the piecewise constant case is needed.

     We shall use the following notation. There exist
 $0=h_{-1}<h_0<h_1<\ldots<h_N=H,$
           and positive constants $c_0,c_1,\ldots,c_N$ so that
           \be\label{eqhjcj-a}
            c(y)=c_{j+1},\,\,\,y\in(h_j,h_{j+1}),\,\,j=-1,0,\ldots,N-1.
           \ee
         We  show that $c(y)$ satisfies the minimal amplitude hypothesis, where the relevant constants depend only \textit{on its total variation}.

         \textbf{Notational comment:} In order to keep the notational uniformity with the other sections, we retain the notation $c_m,\,c_M$ for the minimal and maximal values, respectively, of $c(y).$ Of course they coincide with some $c_j's$ but the distinction in  various estimates  (such as ~\eqref{eqminampconst} below) will be completely clear.

        Recall that $u(\ylambk)$ satisfies Equation  ~\eqref{equation-introduction:1-ng} with
                  $p(\ylambk)= \frac{\lambda}{\mu^2_{k}}- c(y),$ so that

                            \be\label{eqhjpj-new}
            p(\ylambk)=p_{j+1}=\frac{\lambda}{\mu^2_{k}}-c_{j+1},\,\,\,y\in(h_j,h_{j+1}),\,\,
            j=-1,0,\ldots,N-1.
           \ee

         \begin{prop}\label{propcyconst}  Assume  that
          $c(y)$ is piecewise constant as above, and let
            $$V=\suml_{j=0}^{N-1}|c_{j+1}-c_j|$$
            be the total variation of $c(y).$

          Let $u(\ylambk)$ be a normalized solution to ~\eqref{equation-introduction:1-ng}
          where $(\mu_k,\,\lambda)\in \mathscr{A}^{c}_\eps.$

          Then
          \begin{enumerate}
          \item $c(y)$ satisfies the minimal amplitude hypothesis with respect to $\mathscr{A}^{c}_\eps.$

          \item For every $(a,b)\subseteq[0,H]$ there exist  constants $d>0,\,\,$
           $\lambda_0>0,\,$ depending only on $b-a,\,c_m,\,c_M,\,\eps, \,V,$ such that

         \be\label{equy2consta,b}
          \int_a^bu(\ylambk)^2\, dy>d,\quad \lambda>\lambda_0.
                  \ee
         \textbf{Note in particular that $d$ and $\lambda_0$ }do not depend on the size $N$ of the partition.
         \end{enumerate}

         \end{prop}
         \begin{proof}
               In light of Theorem ~\ref{thmcyirreg-minamp} the estimate ~\eqref{equy2consta,b} follows from  the minimal amplitude property.

          Consider an interval $I_j=(h_j,h_{j+1}).$

         The solution $u(\ylambk)$ to ~\eqref{equation-introduction:1-ng} in $I_j$ is given by
         \be\label{equyinIj-a}
         u(\ylambk)=\beta_j\sin\left(\mu_k\sqrt{\frac{p_{j+1}}{c_{j+1}}}(y-\gamma_j)\right),\quad y\in I_j
         \ee
         where $\beta_j,\,\gamma_j$ are suitable constants. Recall from ~\eqref{ineq-intermedzone:2} that $|p_{j+1}|\geq\varepsilon,\,j=-1,0,\ldots ,N-1.$

        In the interval $I_j$ we have
        \be\label{equyinIjbeta}
        u(\ylambk)^2+ c(y)^2 u'(\ylambk)^2\geq\beta_j^2\min(1, c_m\mu_k^2p_{j+1})\geq\beta_j^2\min(1, c_m \mu_k^2\varepsilon).
        \ee
         As $\mu_{k}\geq \mu_{1}>0$, to complete the proof we need to show the existence of a constant $\delta>0,$ depending only on $\eps,\,c_m,\,c_M,\,V$ so that
        \be\label{equyinIj-b}
        |\beta_j|\geq \delta,\quad  j=-1,0,\ldots,N-1.
        \ee
         We observe the following facts concerning the coefficients $\set{\beta_j}_{j=-1}^{N-1}.$
         \begin{itemize}
         \item $$\beta_j \neq 0, \quad j=-1,0,\ldots,N-1,$$
           since otherwise $u(\ylambk)\equiv 0.$
         \item There exists a constant $\kappa>1,$ depending only on
         $c_m,\,c_M,\,\eps, \,V,$  such that
            \be\label{eqsumbeta}\frac{\beta_j^2}{\beta_{j+1}^2},\frac{\beta_{j+1}^2}{\beta_j^2}
            \leq (1 +\kappa |c_{j+2}-c_{j+1}|),\quad j=-1,0,\ldots,N-2.\ee
                     \end{itemize}
   To establish ~\eqref{eqsumbeta} we proceed as follows.

       Denote, for $j=-1,0,\ldots,N-2$

         $$A_{j}=\mu_k\sqrt{\frac{p_{j+1}}{c_{j+1}}}(h_{j+1}-\gamma_j),\,\,\,
         B_{j+1}=\mu_k\sqrt{\frac{p_{j+2}}{c_{j+2}}}(h_{j+1}-\gamma_{j+1}).$$

       %  }
         The continuity of $u(\ylambk)$ and $c(y)u'(\ylambk)$ at $h_{j+1}$ implies that
         \be\label{eqcontumhj}\aligned
         \beta_j\sin(A_{j})=
         \beta_{j+1}\sin(B_{j+1}),\\
         c_{j+1}\beta_j\mu_k\sqrt{\frac{p_{j+1}}{c_{j+1}}}\cos(A_{j})=
         c_{j+2}\beta_{j+1}\mu_k\sqrt{\frac{p_{j+2}}{c_{j+2}}}\cos(B_{j+1}).
         \endaligned\ee
         Recall that $\beta_j\neq 0$ for all $j.$
          It follows that for $ j=-1,0,\ldots,N-2$
          \be\label{eqbetaj2}\aligned
          \beta_j^2=\beta_{j+1}^2\Big(1+\big[\frac{p_{j+2}}{p_{j+1}} \frac{c_{j+2}}{c_{j+1}}-1\big]\cos^2(B_{j+1})\Big),\\
          \beta_{j+1}^2=\beta_{j}^2\Big(1+\big[\frac{p_{j+1}}{p_{j+2}} \frac{c_{j+1}}{ c_{j+2}}-1\big]\cos^2(A_{j})\Big).
          \endaligned\ee
          From ~\eqref{eqhjpj-new} it readily follows that
           \begin{equation}\label{eqpj+1pj}\begin{array}{ll}\begin{array}{lll}
           \frac{p_{j+2}}{p_{j+1}}  \frac{c_{j+2}}{c_{j+1}}-1 &= &\frac{\lambda - \mu_k^2 (c_{j+2} + c_{j+1})}{\lambda-\mu_k^2c_{j+1}}
           \cdot\frac{c_{j+2}-c_{j+1}}{c_{j+1}},\\ 
           \frac{p_{j+1}}{p_{j+2}}\frac{c_{j+1}}{c_{j+2}}-1&=&\frac{\lambda - \mu_k^2 (c_{j+2} + c_{j+1})}
           {\lambda-\mu^2_k c_{j+2}}
           \cdot\frac{c_{j+1}-c_{j+2}}{c_{j+2}},
           \end{array}
            \quad j=-1,0,1,\ldots, N-2.
            \end{array}
           \end{equation}

          The fact that $(\mu_k,\lambda)\in \mathscr{A}^{c}_\eps$ entails
        \be\label{eqestpj}  |\frac{p_{j+1}}{p_{j+2}} \frac{c_{j+1}}{c_{j+2}}-1| , \,\,|\frac{p_{j+2}}{p_{j+1}} \frac{c_{j+2}}{c_{j+1}}-1|\leq
        \kappa |c_{j+1}-c_{j+2}|, \,\,j=-1,0,\ldots,N-2,
                        \ee
                        where $\kappa>0$ depends only on $c_m,\,c_M,\,\eps.$\footnote{In fact, $\kappa = \frac{1}{c_m}\max(1,\frac{c_M -\varepsilon}{\varepsilon})$ works. The delicate point arrives when $c_M < c_{j+2}+c_{j+1}.$ We can eliminate  it since the function $\lambda \to \frac{\lambda - \mu_k^2 (c_{j+2} + c_{j+1})}{\lambda-\mu_k^2c_{j+1}}$ is increasing for $\lambda\geq (c_M +\varepsilon)\mu_k^2.$}
                        In conjunction with ~\eqref{eqbetaj2} the estimate ~\eqref{eqsumbeta} is established.

          From ~\eqref{eqbetaj2} we deduce for any $q\in\set{0,1,\ldots,N-1}$ upper and lower estimates
          $$\begin{cases}\beta_q^2\leq \beta_{-1}^2\prod\limits_{j= 0}^{q}(1+\kappa|c_{j+1}-c_{j}|)\leq \beta_{-1}^2e^{\kappa\suml\limits_{j=0}^{q}|c_{j+1}-c_{j}|},\\
          \beta_q^2\geq\beta_{-1}^2\prod\limits_{j=0}^{q}(1+\kappa|c_{j+1}-c_{j}|)^{-1}\geq \beta_{-1}^2e^{-\kappa\suml\limits_{j=0}^{q}|c_{j+1}-c_{j}|}.
          \end{cases}$$
          Thus, for all $q\in\set{0,1,\ldots,N-1}$, the coefficients $\beta_{-1}$ and $\beta_q$ are comparable in the sense that
          \be\label{eqestbetal}\beta_{-1}^2e^{-\kappa V}\leq\beta_q^2\leq\beta_{-1}^2e^{\kappa V}.\ee

          The normalization of $u(\ylambk)$ in conjunction with \eqref{equyinIj-a} ,~\eqref{eqestbetal} implies
          \be\label{equnormbeta1}1=\int_0^Hu(\ylambk)^2\, dy\leq \suml_{j=-1}^{N-1}\beta_j^2(h_{j+1}-h_j)\leq H\beta_{-1}^2e^{\kappa V}.\ee
          Thus finally the estimate ~\eqref{equyinIj-b} follows from ~\eqref{eqestbetal} and ~\eqref{equnormbeta1}.

          The estimates ~\eqref{equyinIjbeta} and ~\eqref{equyinIj-b} imply that the minimal amplitude hypothesis is satisfied
          \be\label{eqminampconst}
       \mathfrak{r}_{c,\eps}^2= \inf_{\begin{array}{c} y\in \lbrack 0, H\rbrack,\\ (\mu_k,\lambda)\in\mathscr{A}^{c}_\eps\end{array}} \set{u(\ylambk)^2+ (c(y)u'(\ylambk))^2\,\,
      }
       \geq\delta^2\min(1,c_m\mu_1^2\varepsilon)>0,
        \ee
        and $\mathfrak{r}_{c,\eps}$ depends only on $\eps,\,c_m,\,c_M,\,V.$

        The non-concentration estimate ~\eqref{equy2consta,b} is now a consequence of the general Theorem ~\ref{thmcyirreg-minamp}.
         \end{proof}

         We deduce a result for all piecewise constant coefficients having a uniform bound of their total variations.
\begin{cor}\label{corunifcV} Let $\mathscr{K}_{PCV}\subseteq\mathscr{K}$ be the set of all piecewise constant diffusion coefficients, with total variation less than $V.$ Then
\begin{enumerate}
\item \be\label{eqinfrcKV}
\mathfrak{r}_{PCV,\eps}=\inf\limits_{c(y)\in \mathscr{K}_{PCV}}\mathfrak{r}_{c,\eps} >0.
\ee
\item The ratios of the amplitudes of any two waves (between consecutive zeros) are uniformly bounded, for all coefficients in $\mathscr{K}_{PCV}.$
\end{enumerate}
\end{cor}
\begin{proof} The estimate ~\eqref{eqinfrcKV} follows from ~\eqref{eqminampconst}. The second item follows from ~\eqref{eqestbetal}.
\end{proof}

%%%%%%%%%%%%%%%%%%%%%%%%%%%%%%%%%%%%%%%%%%%%%%%%%%%%%%%%%%%%%%%%%
%%%%%%%%%%%%%%%%%%%%%%%%%%%%%%%%%%

%%%%%%%%%%%%%%%%%%%%%%%%%%%%%%%%%%%%%%%%%%%%%%%%%%%%%%%%%%%%%%%%
%%%%%%%%%%%%%%%%%%%%%%%%%%%%%%%%%%%%%%%%%%%%%%%%%%%%%%%%%%%%%%%%

\subsection{\textbf{THE ULTIMATE CASE- $c(y)$ OF BOUNDED VARIATION}}\label{subsecbTV} $\empty$\\

   Our ultimate result concerns the case that $$c(y)\in \mathscr{K}_V=\set{c\in\mathscr{K},\,\,TV(c)\leq V}.$$ 
   Recall that $\mathscr{K}$ was defined in ~\eqref{eq-def-K}.
   
   We establish non concentration  for spectral pairs $(\mu_k,\lambda)\in\mathscr{A}^{c}_\eps$
   (see ~\eqref{eqAceps}).

   As in the cases  above, the proof relies on the validity of the minimal amplitude property, via the fundamental  Theorem ~\ref{thmcyirreg-minamp}.

 \begin{thm}
\label{prop-minam:1-new}
 Let  $c(y)$ be of bounded variation. Then it satisfies the  minimal amplitude hypothesis with respect to $\mathscr{A}^{c}_\eps,$ uniformly for all $\mathscr{K}_V=\set{c(y)\in\mathscr{K},\,\,TV(c)\leq V}.$

   More precisely, as in  ~\eqref{equation-definition-hypo-minam:1},
   
   \begin{itemize}
   \item
  \begin{equation}
 \label{equation-definition-cybv}
 \mathfrak{r}_{c,\eps}^{2} = \inf_{\begin{array}{c} y\in \lbrack 0, H\rbrack,\\ (\mu_k,\lambda)\in\mathscr{A}^{c}_\eps\end{array}}\lbrack u_{\lambda,k}(y)^{2} + (c(y)u'_{\lambda,k}(y))^{2}\rbrack >0,
 \end{equation}
 and
 \begin{equation}
 \label{equation-definition-cybv-unif}
\mathfrak{r}_{V,\eps}^{2}:=\inf\limits_{c\in\mathscr{K}_V} \mathfrak{r}_{c,\eps}^{2} >0.
 \end{equation}
 \item In particular
    \be\label{eqnonconcgeneral-1}
    \inf_{\begin{array}{c} c(y)\in\mathscr{K}_V,\\ (\mu_k,\lambda)\in\mathscr{A}^{c}_\eps\end{array}}\int_a^bu(\ylambk)^2\, dy>0.
    \ee
 \item For the asymptotic behavior of the eigenfunctions, as $\lambda\to\infty,$ we have the following statement concerning the lower bound.\\
    For any $(a,b)\subseteq (0,H)$ there exist a constant $\lambda_0$ depending only on $\eps,\,c_m,\,c_M,\,b-a,\,\mathfrak{r}_{V,\eps}$ and a constant $\mathfrak{f}>0$ depending only on $\eps,\,c_m,\,c_M,\,b-a$ such that, for every $c(y)\in\mathscr{K}_V,$ and  for every normalized $u(\ylambk)$ associated with $(\mu_k,\lambda)\in \mathscr{A}^{c}_\eps,$
    \be\label{eqnonconcgeneral}
    \int_a^bu(\ylambk)^2\, dy\geq \mathfrak{f},\quad \lambda>\lambda_0,\,\,c(y)\in\mathscr{K}_V.
    \ee

     \item 
    Finally, let $\omega=\omega'\times(a,b).$ If $\omega'\neq\Omega'$ assume that the family $\set{\phi_k(x')}_{k=1}^\infty$ of eigenfunctions of the Laplacian in $\Omega'\subseteq\RR^d$ does not concentrate in $\Omega'\setminus\omega'.$
             \\
             Then there exists a constant $\mathfrak{g}>0$ such that any eigenfunction $v_{\lambda}(x',y)=u(\ylambk)\phi_k(x')$ satisfies
      \be\label{eqnoncontrateab-1}
      \mathfrak{g}\leq\frac{\Vert v_{\lambda}\Vert_{L^{2}(\omega)}}{\Vert v_{\lambda}\Vert_{L^{2}(\Omega)}}\leq 1
      \ee
      uniformly for all $c(y)\in\mathscr{K}_V$ and all eigenvalues in $ \mathscr{A}^{c}_\eps.$
      \end{itemize}
      
 \end{thm}

                The proof consists of approximating $c(y)$ by a sequence of piecewise constant functions and using the results of Proposition ~\ref{propcyconst} and Corollary ~\ref{corunifcV}. The approximation procedure is based on the following result  ~\cite[pp. 12-13]{Bres:1}.
   \begin{claim}
\label{claim:4-new}
Suppose that $c(y)\in\mathscr{K}$  and is of total variation $V>0.$
 Then there exists a sequence of piecewise constant functions $\set{c^{(n)}(y)}_{n=1}^\infty,$ so that

 \begin{itemize}
 \item \begin{equation}
 \label{constantBV:1-unif}
 \lim\limits_{n\to\infty} c^{(n)}(y)=c(y),\quad \mbox{uniformly in }\,\,y\in[0,H]. \ee
 \item \begin{equation}
 \label{constantBV:1-new}
 \sup \set{ TV(c^{(n)})}_{n=1}^\infty \leq V= TV(c).\ee
 \item \begin{equation}
 \label{constantBV:2-new}
 0<c_m={\rm ess}\, \inf \,\,c(y)\leq c_m^{(n)}={\rm ess}\, \inf\, c^{(n)}(y)\leq c_M^{(n)}={\rm ess}\,\sup \,\, c^{(n)}(y)\leq c_M={\rm ess}\,\sup \,\,c(y),\quad n=1,2,\ldots\ee
 \end{itemize}

\end{claim}
 Recall (see Introduction) that we denote
 $$\widetilde{c}(x',y)=c(y),\,\,\,\,\widetilde{c^{(n)}}(x',y)=c^{(n)}(y),\quad x=(x',y)\in \Omega=\Omega'\times[0,H],$$
 with associated operators
$$A=-\nabla\cdot(\widetilde{c}\,\nabla),\,\,\quad A^{(n)}=-\nabla\cdot(\widetilde{c^{(n)}}\,\nabla).$$
 For the Laplacian $-\Delta_{x'}$ acting in $L^{2}(\Omega')$ with domain $H^2(\Omega')\cap H^1_0(\Omega')$, we denote by $(\mu_k^2,\phi_k)_{k\geq 1}$ the sequence of normalized eigenfunctions and their associated eigenvalues, ordered by $\mu_k\leq \mu_{k+1}.$ \\
 The eigenfunctions of $A$ (resp. $A^{(n)}$) are
 $$v_\lambda(x)=u(y;\lambda,k)\phi_k(x'),\quad v_\lambda^{(n)}(x)=u^{(n)}(y;\lambda,k)\phi_k(x'),$$
 where $u(y;\lambda,k)$ (resp. $u^{(n)}(y;\lambda,k)$) satisfies Equation
         ~\eqref{equation-introduction:1-ng} (resp. ~\eqref{equation-introduction:1-ng} with $c,\,p$ replaced by $c^{(n)},\,p^{(n)}$) and they are normalized in $L^{2}((0,H), dy).$ 

         We consider eigenfunctions associated with spectral pairs  $(\mu_k,\lambda)\in\mathscr{A}^{c}_\eps.$

           The following perturbation lemma  is at the basis of the proof of the theorem. We postpone its proof to the end of this section, following the proof of the theorem. Note that in this lemma no assumption is needed concerning the total variations of the involved functions.

         \begin{lem}[ \textbf{Convergence of eigenvalues and eigenfunctions}]
\label{lem-eigenfunction:1-new}
Let $\set{c^{(n)}}_{n=1}^\infty\subseteq \mathscr{K}$ converge uniformly (in $[0,H]$) to $c(y).$ Let $A^{(n)}$ and $A$ be the corresponding operators. Let $\lambda>0$ be an eigenvalue of $A,$ and let $k\in \mathbb{N}^\ast$ so that $(\mu_k,\lambda)\in\mathscr{A}^{c}_\eps.$ Let $u(y;\lambda,k)\phi_k(x')$ be the associated normalized eigenfunction.  Then there exist $N>0$ and a sequence of eigenvalues $\set{\lambda^{(n)}}_{n=N}^\infty$ of $\set{A^{(n)}}_{n=N}^\infty$ (in fact, of $\set{A^{(n)}_k}_{n=N}^\infty$), with associated normalized eigenfunctions $\set{u^{(n)}(y;\lambda^{(n)},k)\phi_{k}(x')}_{n=N}^\infty$ such that
  \begin{equation}
  \label{eq-conv-eigen}
  \begin{array}{l}
  (i)\hspace{50pt} \,\lim_{n\to\infty}\lambda^{(n)}=\lambda,\\
   (ii)\hspace{48pt} \lim_{n\to\infty}u^{(n)}(\cdot;\lambda^{(n)},k)
   =u(\cdot;\lambda,k) \,\,\mbox{uniformly in}\,\,[0,H],  \\ 
   (iii)\hspace{48pt}\lim_{n\to\infty}c^{(n)}(y)( u^{(n)})'(\cdot;\lambda^{(n)},k)
   =c(y) u'(\cdot;\lambda,k) \,\,\mbox{uniformly in}\,\,[0,H].
   \end{array}
  \end{equation}
 \end{lem}

\begin{proof}[\textbf{Proof of Theorem ~\ref{prop-minam:1-new}}] Pick an eigenvalue $\lambda \geq (c_M+\eps)\mu_k^2.$

Let $\set{c^{(n)}(y)}_{n=1}^\infty$ be a sequence of piecewise constant functions as in Claim ~\ref{claim:4-new}. Let $\set{u^{(n)}(y;\lambda^{(n)},k)}_{n=N}^\infty$ be the corresponding sequence as in Lemma ~\ref{lem-eigenfunction:1-new}. Note that the convergence ~\eqref{eq-conv-eigen}(i) implies that, for sufficiently large index $N$ the condition $\lambda^{(n)}>(c_M+\frac{\eps}{2})\mu_{k}^2$ holds for $n>N.$ In view of the uniform bound ~\eqref{constantBV:1-new} on total variations we can invoke Corollary ~\ref{corunifcV} to get
\be\label{equnifrnj}
(\mathfrak{r}_{c,\eps}^{app})^2:= \inf_{N\leq n<\infty}\inf_{y\in [0,H]}\Big(u^{(n)}(y;\lambda^{(n)},k)^2+ (c^{(n)}(y)u^{(n)'}(y;\lambda^{(n)},k))^2\Big)>S^2,
\ee
where $S>0$ depends only on $\eps,\,V,\,c_m,\,c_M$ (see ~\eqref{eqminampconst}).

%%%%%%%%%%%%%%%%%%%%%%%%%%%%%%%%%%%%%%%%%%%%%%%%%%%%%%%%%%%%%%%%%%
%%%%%%%%%%%%%%%%%%%%%%%%%%%%%%%%%%%%%%%%%%%%%%%%%%%%%%%%%%%%%%%%%%
%%%%%%%%%%%%%%%%%%%%%%%%%%%%%%%%%%%%%%%%%%%%%%%%%%%%%%%%%%%%%%%%%%
The convergence result in Lemma ~\ref{lem-eigenfunction:1-new} entails uniform convergence of both the functions and their derivatives, hence
\be\label{eqinfuymuk}u(y;\lambda,k)^2+(c(y)u'(y;\lambda,k))^2\geq S^2,\quad y\in[0,H].
\ee
The estimate ~\eqref{equation-definition-cybv} now follows from the fact that, in view of Corollary ~\ref{corunifcV}, the estimate ~\eqref{equnifrnj} holds uniformly for all approximating sequences for any solution $u(y;\lambda,k)$ associated with $(\mu_k,\lambda)\in \mathscr{A}^{c}_\eps.$ In fact, we get the uniform estimate  ~\eqref{equation-definition-cybv-unif} since $\mathfrak{r}_{c,\eps}$ depends only on $V.$

We now turn to the non-concentration statement ~\eqref{eqnonconcgeneral}. The general Theorem ~\ref{thmcyirreg-minamp} ensures the existence of $d>0,\,\lambda_0>0$ depending on $\eps,\,c_M,\,c_m,\, b-a,\,\mathfrak{r}_{c,\eps},$          such that

              $$  \int_{a}^{b}\, u(y;\lambda,k)^2\,dy\geq d,\quad
                 \lambda>\lambda_0.$$
                 Due to ~\eqref{equation-definition-cybv-unif} this estimate is uniformly valid (with the same $d,\lambda_0$) for all $c(y)\in\mathscr{K}$ such that $TV(c)\leq V.$

However, for every $c(y)\in\mathscr{K}$ with $TV(c)\leq V$ there are finitely many eigenfunctions that are excluded, namely, those with $\lambda<\lambda_0.$ Clearly, these eigenfunctions vary with $c(y).$ We now show that they can be included in the estimate ~\eqref{eqnonconcgeneral-1}. The price to be paid is that the lower bound $\mathfrak{g}$ depends in a more delicate way on the various parameters (and not only on $\eps,\,c_M,\,c_m,\, b-a,\,\mathfrak{r_{c,\eps}}$).

To obtain a contradiction let $\set{c^{(n)}(y),\,\,TV(c^{(n)})\leq V}_{n=1}^\infty\subseteq\mathscr{K}.$
Let $\set{\lambda_n}_{n=1}^\infty$ be a sequence of eigenvalues with associated normalized eigenfunctions $\set{u_n(y;\lambda_n,k)}_{n=1}^\infty$ satisfying
  \be\label{eqnonguide-n1}
        (c^{(n)}(y)u_n'(y;\lambda_n,k))'+\Big(\lambda_n - c^{(n)}(y)\mu^2_{k}\Big)u_n(y;\lambda_n,k)=0.
\ee
  Assume further that
  $$(c_M+\eps)\mu^2_{k}<\lambda_n<\lambda_0,\quad n=1,2,\ldots$$
  Suppose that for some interval $(a,b)\subseteq (0,H)$ and some subsequence (we do not change indices)
  \be\label{equntozero}
  \lim\limits_{n\to\infty} \int_a^bu_n(y;\lambda_n,k)^2dy=0.
  \ee
  The sequence is normalized and clearly the coefficients $\set{\lambda_n - c^{(n)}(y)\mu^2_{k},\,\,n=1,2,\ldots}$ are uniformly bounded, hence from ~\eqref{eqnonguide-n1} we infer
  \be\label{equntozerou'}
  \lim\limits_{n\to\infty} \int_a^b\Big(c^{(n)}(y)u_n'(y;\lambda_n,k)\Big)^2dy=0.
  \ee
  This is a contradiction to the fact (see ~\eqref{equation-definition-cybv-unif})
  $$u_n(y;\lambda_n,k)^2+(c^{(n)}(y)u_n'(y;\lambda_n,k))^2\geq \mathfrak{r}_{V,\eps}^{2},\,\,y\in [0,H],\quad n=1,2,\ldots$$
The estimate ~\eqref{eqnoncontrateab-1} follows from ~\eqref{eqnonconcgeneral-1}.
\end{proof}

%%%%%%%%%%%%%%%%%%%%%%%%%%%%%%%%%%%%%%%%%%%%%%%%%%%%%%%%%%%%%%%%%
%%%%%%%%%%%%%%%%%%%%%%%%%%%%%%%%%%%%%%%%%%%%%%%%%%%%%%%%%%%%%%%%%%

%%%%%%%%%%%%%%%%%%%%%%%%%%%%%%%%%%%%%%%%%%%%%%%%%%%%%%%%%%%%%%%%%%
%%%%%%%%%%%%%%%%%%%%%%%%%%%%%%%%%%%%%%%%%%%%%%%%%%%%%%%%%%%%%%%%%%
%%%%%%%%%%%%%%%%%%%%%%%%%%%%%%%%%%%%%%%%%%%%%%%%%%%%%%%%%%%%%%%%%%
\begin{proof}
[\textbf{Proof of Lemma ~\ref{lem-eigenfunction:1-new}}]
We use the direct sum representation ~\eqref{equation-spectralstructure:1} both for the operator $A$ and the operators $A^{(n)}.$ Since the eigenfunctions $\set{\phi_k(x')}$ do not depend on the index $n,$ the reduced operators  $A_k^{(n)}$ (see ~\eqref{eqAkofy}) are given by

$$A_{k}^{(n)}=c^{(n)}(y) \mu_{k}^{2} - \frac{d}{dy}c^{(n)}(y)\frac{d}{dy}, k = 1,2,\ldots,\mbox{ with }D(A_k^{(n)})=\set{u\in H^1_0(0,H), c^{(n)}u'\in H^1(0,H)},\quad n=1,2,\ldots.
$$
Fix $k\in \mathbb{N}^*$ so that $\lambda$ is an eigenvalue of $A_k.$ The corresponding (reduced) eigenfunction $u(\ylambk)$ satisfies the equation (see ~\eqref{equation-introduction:1-ng})
$$
 (c(y)\mu_k^2-\lambda)u(\ylambk) - \frac{d}{dy}c(y)\frac{d}{dy}u(\ylambk)= 0, \quad u(\ylambk)(0) = u(\ylambk)(H) = 0.
$$
Let $B(\lambda,\delta)\subseteq  \mathbb{C}$ be the disk of radius $\delta$ centered at $\lambda$
   and consider the following linear initial value problem, with a complex parameter $z\in {\overline{B(\lambda,\delta)}},$
  \be\label{eqwinitz-a}
  \begin{cases}(c(y)w'(y;z))' + (z-c(y)\mu^{2}_{k})w(y;z)= 0,\\
   w(0;z) = 0,\,\,c(0)w'(0;z)=c(0)u'(0;\lambda,k).
   \end{cases}
  \ee
     Denoting $w_1(y;z)=w(y;z),\,\,w_2(y;z)=c(y)w_1'(y;z)$ we have
     \be\label{eqw1w2}
     \left(
       \begin{array}{c}
         w_1 \\
         w_2\\
       \end{array}
     \right)'=\left(
                \begin{array}{cc}
                  0 & c(y)^{-1} \\
                  -(z-c(y)\mu_k^2) & 0 \\
                \end{array}
              \right)\cdot\left(
                            \begin{array}{c}
                              w_1 \\
                              w_2 \\
                            \end{array}
                          \right).
     \ee

      For every $y\in [0,H]$ the function $w(y;z)$ is analytic as a function of $z$  ~\cite[Chapter 1, Th.8.4]{coddington} and this is true in particular for $f(z):=w(H;z).$ Note that $z$ is an eigenvalue of $A$ if and only if $f(z)=0,$
     since  if $w$ is an eigenfunction then so is $aw,$ for any $a\neq 0.$ Clearly $f(\lambda)=0.$ This is the only zero of $f$ in  $B(\lambda,\delta)$ for sufficiently small $\delta>0,$ since $\lambda$ is an isolated eigenvalue of
      $A_k.$

      By  standard formulas for zeros of analytic functions, since $\lambda$ is a simple zero,
      \be\label{eqf'flambda}
       1=\frac{1}{2\pi i}\int\limits_{|z-\lambda|=r}\frac{f'(z)}{f(z)}dz,\quad \lambda=\frac{1}{2\pi i}\int\limits_{|z-\lambda|=r}z\frac{f'(z)}{f(z)}dz.
      \ee

      Replacing in ~\eqref{eqwinitz-a} $c(y)$ by  $c^{(n)}(y)$ (and retaining the initial conditions) we obtain solutions $w^{(n)}(y;z).$ From  Equation ~\eqref{eqw1w2} and the boundedness of the coefficients we infer that the functions 
            $$\set{(w^{(n)})'(y;z),\,\,\frac{d}{dy}(c^{(n)}(y){w^{(n)}}'(y;z)),\,\,z\in {\overline{B(\lambda,\delta)}} }_{n=1}^\infty$$ are uniformly bounded   in $y\in[0,H].$ Consequently the functions
              $\set{w^{(n)}(y;z),\,\,c^{(n)}(y)(w^{(n)})'(y;z),\,\,z\in {\overline{B(\lambda,\delta)}} }_{n=1}^\infty$ are uniformly Lipschitz continuous for $y\in [0,H].$
       
        Pick $z\in B(\lambda,\delta)$ and  let $\set{z^{(n)}}_{n=1}^\infty\subseteq B(\lambda,\delta)$ be a sequence converging to $z.$ The Arzela-Ascoli  theorem yields the existence of a subsequence $\set{w^{(n_j)}}_{j=1}^\infty$ and a limit function $\tilde{w}(y;z)$ such that
      \be\label{eqconvwnj}
      \lim\limits_{j\to\infty}w^{(n_j)}(y;z^{(n_j)})=\tilde{w}(y;z),\quad \lim\limits_{j\to\infty}c^{(n_j)}(y)w^{(n_j)'}(y;z^{(n_j)})=c(y)\tilde{w}'(y;z),\,\,\mbox{uniformly in}\,\,[0,H].
      \ee
         It follows that $\tilde{w}(y;z)$ satisfies  (first in distribution sense) the equation  ~\eqref{eqwinitz-a} with the same initial data, so by uniqueness
        $\tilde{w}(y;z)=w(y;z).$ In particular, since  all converging  subsequences have the same limit,~\eqref{eqconvwnj} can be replaced by
         \be\label{eqconvwnn}
      \lim\limits_{n\to\infty}w^{(n)}(y;z^{(n)})=w(y;z),\quad \lim\limits_{n\to\infty}c^{(n)}(y)w^{(n)'}(y;z^{(n)})=c(y)w'(y;z),\,\,\mbox{uniformly in}\,\,[0,H],
      \ee

         Setting $f^{(n)}(z)=w^{(n)}(H;z)$ we obtain from ~\eqref{eqf'flambda}, in view of the convergence ~\eqref{eqconvwnn} and for sufficiently large $n,$
          \be\label{eqfn'fnlambda}
       1=\frac{1}{2\pi i}\int\limits_{|z-\lambda|=r}\frac{{f^{(n)}}'(z)}{f^{(n)}(z)}dz,\quad \lambda^{(n)}=\frac{1}{2\pi i}\int\limits_{|z-\lambda|=r}z\frac{{f^{(n)}}'(z)}{f^{(n)}(z)}dz,
      \ee
          where
           $\lim\limits_{n\to\infty}\lambda^{(n)}=\lambda.$

          In addition, $\lambda^{(n)}$ (for sufficiently large $n$) is an eigenvalue of $A_k^{(n)}$ (and in particular is real) and $w^{(n)}(y;\lambda^{(n)})$ is an associated (not necessarily normalized) eigenfunction.
          To conclude the proof of the lemma we take $$u^{(n)}(y;\lambda^{(n)},k)=
          \frac{w^{(n)}(y;\lambda^{(n)})}{\Big[\int_0^H|w^{(n)}(y;\lambda^{(n)})|^2dy\Big]^\frac12}.$$

\end{proof}

%%%%%%%%%%%%%%%%%%%%%%%%%%%%%%%%%%%%%%%%%%%%%%%%%%%%%%%%%%%%%%%%%%%
%%%%%%%%%%%%%%%%%%%%%%%%%%%%%%%%%%%%%%%%%%%%%%%%%%%%%%
%%%%%%%%%%%%%%%%%%%%%%%%%%%%%%%%%%%%%%%%%%%%%%%%%%%%%%

\end{document}